\documentclass[11pt,twoside,reqno]{amsart}

\usepackage{mathtools}
\usepackage{todonotes}
\usepackage{microtype}
\usepackage[noadjust]{cite}
\usepackage[OT1]{fontenc}
\usepackage{type1cm}
\usepackage{amssymb}
\usepackage{dsfont}
\usepackage{enumitem}
\usepackage{comment}
\usepackage{xcolor}
\usepackage{lmodern}

\usepackage{geometry}
\usepackage{hyperref}
\hypersetup{
  colorlinks=true,
  linkcolor=black,
  anchorcolor=black,
  citecolor=black,
  filecolor=black,      
  menucolor=red,
  runcolor=black,
  urlcolor=black,
}
\usepackage[capitalize,noabbrev,nameinlink]{cleveref}

\makeatletter
\AddToHook{cmd/appendix/before}{\def\cref@section@alias{appendix}}
\makeatother

\theoremstyle{plain}
\newtheorem{theorem}{Theorem}[section]

\newtheorem{corollary}[theorem]{Corollary}
\newtheorem{proposition}[theorem]{Proposition}

\newtheorem{lemma}[theorem]{Lemma}

\theoremstyle{remark}
\newtheorem{remark}[theorem]{Remark}

\theoremstyle{definition}

\newtheorem*{definition*}{Definition}

\newtheorem*{question*}{Question}

\AddToHook{env/proposition/begin}{\crefalias{theorem}{proposition}}
\AddToHook{env/lemma/begin}{\crefalias{theorem}{lemma}}
\AddToHook{env/corollary/begin}{\crefalias{theorem}{corollary}}
\AddToHook{env/claim/begin}{\crefalias{theorem}{claim}}
\AddToHook{env/conjecture/begin}{\crefalias{theorem}{conjecture}}
\AddToHook{env/remark/begin}{\crefalias{theorem}{remakr}}
\AddToHook{env/example/begin}{\crefalias{theorem}{example}}
\AddToHook{env/remark/begin}{\crefalias{theorem}{remark}}

\crefformat{equation}{(#2#1#3)}
\crefformat{figure}{#2Figure #1#3}
\crefformat{enumi}{(#2#1#3)}
\crefformat{theorem}{#2Theorem~#1#3}
\crefformat{lemma}{#2Lemma~#1#3}
\crefformat{proposition}{#2Proposition~#1#3}
\crefformat{corollary}{#2Corollary~#1#3}
\crefformat{claim}{#2Claim~#1#3}
\crefformat{section}{#2Section~#1#3}
\crefformat{subsection}{#2Section~#1#3}
\crefformat{remark}{#2Remark~#1#3}

\numberwithin{equation}{section}

\newcommand{\rrr}{{\underline{r}}}

\renewcommand{\AA}{\mathcal{A}}

\newcommand{\PP}{\mathcal{P}}
\newcommand{\PPC}{\mathcal{P}_{\mathrm{c}}}

\newcommand{\R}{\mathbb{R}}

\newcommand{\Z}{\mathbb{Z}}
\newcommand{\N}{\mathbb{N}}

\newcommand{\dd}{\,\mathrm{d}}
\newcommand{\mtt}[1]{\mathtt{#1}}

\renewcommand{\geq}{\geqslant}
\renewcommand{\leq}{\leqslant}

\DeclareMathOperator{\ldimloc}{\underline{dim}_{loc}}

\DeclareMathOperator{\dimh}{dim_H}

\DeclareMathOperator{\capa}{\mathrm{Cap}}

\DeclareMathOperator{\esssup}{ess\,sup}

\DeclareMathOperator{\dist}{dist}
\DeclareMathOperator{\diam}{diam}

\DeclareMathOperator{\spt}{spt}

\usepackage{tabularx}
\usepackage{booktabs}

\begin{document}

\title[Dvoretzky covering problem for general measures]{Dvoretzky covering problem for general measures}

\author{Roope Anttila}
\address[Roope Anttila]
        {University of St Andrews \\ 
         Mathematical Institute\\ 
         St Andrews KY16 9SS\\ 
         Scotland}
\email{ranttila97@gmail.com}

\author{Markus Myllyoja}
\address[Markus Myllyoja]
        {Research Unit of Mathematical Sciences \\ 
         P.O.\ Box 8000 \\ 
         FI-90014 University of Oulu \\ 
         Finland}
\email{markus.myllyoja@oulu.fi}

\date{\today}
\thanks{The authors thank Esa Järvenpää and Ville Suomala for reading an earlier draft of the paper and for providing many useful comments and suggestions. RA was financially supported by the Magnus Ehrnrooth Foundation and EPSRC, grant no. EP/Z533440/1. MM was supported by the Emil Aaltonen Foundation.}

\begin{abstract}
    We study the Dvoretzky covering problem for random covering sets driven by general Borel probability measures. As our main result, we solve the problem of covering analytic sets by random covering sets generated by arbitrary Borel probability measures on the real line. Prior to this work, a complete solution was not known for any singular measure. Our solution is potential theoretic and involves a generalisation of a notion of capacity in the work of Kahane, who solved the problem of covering compact sets in the classical setting where the random covering process is driven by the Lebesgue measure on the unit circle. One of our key innovations is a simple but powerful application of the Jankov-von Neumann uniformisation theorem, which we believe to have interest outside of this work.
    
    In addition, we determine the critical exponent for the covering problem for polynomially decreasing sequences $(cn^{-t})_n$ for random covering sets driven by Borel probability measures on $\R^d$. At exactly the critical exponent, the covering property generally depends on the constant $c>0$, and as an application of our main result, we determine the critical constant for random covering sets driven by natural measures on strongly separated self-conformal sets on the line. The critical constant depends on the multifractal structure of the average densities of the measure, and the result is new even for the simplest case of the Hausdorff measure on the Cantor set.
\end{abstract}

\maketitle

\section{Introduction}\label{sec:intro}

Let $\mu$ be a Borel probability measure on $\R^d$ and let $\rrr = (r_n)_n$ be a sequence of positive real numbers. Consider the probability space $(\Omega,\mathbb{P}_{\mu})$, where $\Omega\coloneqq (\R^d)^{\N}$ and $\mathbb{P}_{\mu}\coloneqq \mu^\N$. Given $\omega\in\Omega$, we denote by
\begin{equation*}
    E_{\rrr}(\omega)\coloneqq \limsup_{k\to\infty}B(\omega_k,r_k)= \bigcap_{k=1}^{\infty}\bigcup_{n=k}^{\infty}B(\omega_n,r_n),
\end{equation*}
the \emph{random covering set} corresponding to the sequence of radii $\rrr$. Here and hereafter, $B(x,r)$ denotes an open ball with centre $x\in\R^d$ and radius $r>0$, although the main results of this paper hold for random covering sets generated by closed balls with minor modifications to the proofs. As usual, we sometimes drop $\omega$ from the notation of the random variable and denote $E_{\rrr}(\omega)$ simply by $E_{\rrr}$.

A fundamental question in the study of random covering sets is the \emph{Dvoretzky covering problem} which asks for characterisations for when a given (measurable) set $A\subset \R^d$ is $\mathbb{P}_{\mu}$-almost surely covered by $E_{\rrr}$. This question is sensible since for any measurable set $A$, the $\mathbb{P}_{\mu}$-probability of covering $A$ by $E_{\rrr}$ is always either zero or one by Komogorov's zero-one law. Dvoretzky's problem turns out to be quite subtle already in the classical setting where $\mu$ is the Lebesgue measure on the unit circle, often parametrised as $\mathbb{T}\coloneqq \R/\Z$ in modern formulations. In this setting, the question for covering the full circle was solved in its full generality in 1972 by Shepp in a seminal paper \cite{Shepp1972}, where he showed that the condition
\begin{equation}\label{eq:shepp}
    \sum_{n=1}^{\infty}\frac{1}{n^2}\exp\left(\sum_{k=1}^n2r_n\right)=\infty,
\end{equation}
is necessary and sufficient for covering $\mathbb{T}$ by $E_{\rrr}$ almost surely. This result was later generalised by Kahane \cite{Kahane1990}, who gave a potential theoretic characterisation, using a suitable notion of capacity, for almost surely covering a compact subset of $\mathbb{T}$. In addition to the classical setting, a solution to the Dvoretzky covering problem for random covering sets generated by balls was, prior to this work, only known in the case where the centres are distributed according to an absolutely continuous measure on the unit circle, but even then, one requires a regularity assumption on the density of the measure \cite{FanFriend2021}. Our main contribution in this article is the solution of the Dvoretzky covering problem for arbitrary Borel probability measures on $\R$ (or on the unit circle). Our solution involves a generalisation of Kahane's notion of capacity and we are able to characterise the covering of analytic subsets of $\R$, providing a slight improvement on the state of the art even in the classical setting of the Lebesgue measure on the unit circle by upgrading Kahane's result from compact to analytic sets. We note that the classical theory is focused on the Lebesgue measure on $\mathbb{T}$ instead of on $[0,1]$, say, since the rotational invariance of the Lebesgue measure on $\mathbb{T}$ simplifies the situation slightly. By working with general measures on $\mathbb{T}$ we would already lose this simplification so, to us, it is more natural to work with measures on $\R$ instead. This also introduces the additional complexity that the support of the measure $\mu$ need not be bounded. It should be noted however, that all of our results hold for arbitrary Borel probability measures on $\mathbb{T}$ just as well.

\subsection{Random covering sets for the Lebesgue measure}
The study of random covering sets can be traced back to the late 1800s, where Borel made the following observation \cite{Borel1897}: If one places a sequence of random arcs with predetermined lengths on the unit circle, with centres chosen independently and uniformly at random, then any given point on the circle is almost surely covered by infinitely many of the arcs if the sum of their lengths diverges. This observation can be seen as the origin of the Borel--Cantelli lemma, which is now a standard tool in probability theory, and even though the lemma is left implicit in Borel's article \cite{Borel1897}, it is the first instance where the idea of this lemma appears in writing. Identifying the unit circle with $\mathbb{T}= \R/\Z$, Borel's observation corresponds to saying that if $\mu=\mathcal{L}$ is the Lebesgue measure on $\mathbb{T}$,  $\rrr=(r_n)_n$ is a sequence of radii and $\ell_n=2r_n$, then the condition 
\begin{equation}\label{eq:lengths-diverge}
    \sum_{n=1}^{\infty}\ell_n=\infty,
\end{equation}
implies that a given point $x\in\mathbb{T}$ is covered by $E_{\rrr}$ almost surely, that is $\mathbb{P}(x\in E_{\rrr})=1$. In fact, by using the Borel--Cantelli lemma, the condition \cref{eq:lengths-diverge} is easily seen to be a necessary and sufficient condition for covering a single point, since $\mathbb{P}(x\in B(\omega_n,r_n))=\mathbb{P}(\omega_n\in B(x,r_n))=\ell_n$ and the centres are chosen independently. Combining this with a simple application of Fubini's theorem yields the following dichotomy for the Lebesgue measure of $E_{\rrr}$: With probability $1$, 
\begin{equation*}\label{eq:possibleuselesseq}
    \mathcal{L}(E_{\rrr}(\omega))=\begin{cases}
        1,& \text{ if } \sum_n \ell_n=\infty,\\
        0,& \text{ if } \sum_n \ell_n<\infty. 
    \end{cases}
\end{equation*}
Thus the condition \cref{eq:lengths-diverge} also gives a complete description of the size of $E_{\rrr}$ in term of its Lebesgue measure. This rather simple observation opens the way for many interesting follow-up questions. For example, if the condition \cref{eq:lengths-diverge} does not hold, then the random covering set will be of zero Lebesgue measure almost surely, in which case it is natural to ask about its (Hausdorff) dimension. This question together with its analogues in more general settings form an active area of study; for various results in this direction, we refer the reader to  \cite{JarvenpaasMyllyojaStenflo2025,JarvenpaaMyllyojaSeuret2025+,FengJarvenpaasSuomala2018,EkstromPersson2018,Durand2010,FanSchmelingTroubetzkoy2013,FanWu2004,Daviaud025,Daviaud2026+, Myllyoja2024, Seuret2018, JarvenpaasKoivusaloLiSuomala2014, EkstromJarvenpaas2020, Persson2015}.

Some 50 years after Borel's work, Aryeh Dvoretzky \cite{Dvoretzky1956} asked the following question which is the original formulation of the Dvoretzky covering problem: Is \cref{eq:lengths-diverge} sufficient for covering the full circle $\mathbb{T}$ by $E_{\rrr}$ almost surely? As seen above, \cref{eq:lengths-diverge} implies that any given point is covered almost surely and this immediately extends to covering countable sets, but beyond that a more careful study is required.
Dvoretzky gave the first non-trivial sufficient condition for covering the full circle, namely if the lengths decrease slower than $\frac{\log n}{n}$, then the circle is covered almost surely, but perhaps surprisingly, he also answered his question in the negative by constructing a sequence $\rrr=(r_n)_n$ which satisfies \cref{eq:lengths-diverge}, but nevertheless $\mathbb{P}(\mathbb{T}\subset E_{\rrr})=0$. 

Dvoretzky's foundational result essentially started the study of random covering sets and it was due to his article that the question of finding characterisations for when the full circle, or more generally a given subset of the circle, is covered came to be known as the Dvoretzky covering problem. After Dvoretzky's result, the problem was first studied in the special case of \emph{polynomially decreasing sequences of radii} $\rrr=(cn^{-t})_n$, with $c,t>0$. In this setting, it follows from Borel's and Dvoretzky's results, respectively, that if $t>1$, then $\mathbb{T}$ is not covered almost surely and if $t<1$, then $\mathbb{T}$ is covered almost surely, irrespective of the value of the constant $c>0$. In the case of the critical exponent $t=1$, the covering property is more subtle and does depend on the constant $c>0$. Kahane \cite{Kahane1959} showed that if $\rrr=(cn^{-1})_n$ for $c>\frac{1}{2}$, then $\mathbb{T}$ is covered almost surely and in the case $c<\frac{1}{2}$, Billard \cite{Billard1965} showed that we almost surely do not have a covering. The remaining case of $c=\frac{1}{2}$ was solved by Erdös already in 1961 in an unpublished work, and a proof that this case is a covering case was published later by Mandelbrot \cite{Mandelbrot1972}. Soon after this, Shepp published his seminal article \cite{Shepp1972}, which solved the problem of covering the full torus by showing that \cref{eq:shepp} characterises covering.

We call the notion of capacity which Kahane used in his generalisation of Shepp's result for compact sets the \emph{$\rrr$-capacity}. We refer to \cref{sec:main_results} for the precise definition but note here that $\capa_{\rrr}(A)=0$ if and only if the \emph{$\rrr$-energy} of all Borel probabilty measures $\nu$ on $A$ is infinite, that is,
\begin{equation*}
    I_{\rrr}(\nu)\coloneqq\iint\exp\left(\sum_{n=1}^{\infty}\max\{2r_n-|x-y|,0\}\right)\dd\nu(x)\dd\nu(y)=\infty
\end{equation*}
for all Borel probability measures $\nu$ whose support is contained in $A$. The precise formulation of Kahane's generalisation of Shepp's  result is as follows \cite{Kahane1990}.
\begin{theorem}[Kahane 1990]\label{thm:kahane}
    Let $\rrr=(r_n)_n$ be a non-increasing sequence tending to zero and let $A\subset \mathbb{T}$ be a non-empty compact set. Then $A$ is covered almost surely by $E_{\rrr}$ if and only if
    \begin{equation*}
        \capa_{\rrr}(A)=0.
    \end{equation*}
\end{theorem}
Shepp's condition \cref{eq:shepp} turns out to be equivalent with
\begin{equation}\label{eq:shepp-equiv}
    \iint\exp\left(\sum_{n=1}^{\infty}\max\{2r_n-|x-y|,0\}\right)\dd \mathcal{L}(x)\dd  \mathcal{L}(y)=\infty,
\end{equation}
that is with the $\rrr$-energy of the Lebesgue measure being infinite. In fact, Shepp's approach for the proof of \cref{eq:shepp} is based on showing that the divergence of an integral similar to \cref{eq:shepp-equiv} characterises the covering property, and then showing that the divergence of the integral is equivalent with \cref{eq:shepp}.

\subsection{Random covering sets for general measures}
The progress described thus far, culminating in the work of Kahane, provides a satisfactory solution to the Dvoretzky covering problem on the circle $\mathbb{T}$, when the random covering sets are formed by balls with uniformly distributed centres. While other geometric properties of random covering sets, such as their Hausdorff dimension, are quite well understood in very general settings---both for general measures on Euclidean spaces \cite{EkstromPersson2018,FengJarvenpaasSuomala2018,JarvenpaasMyllyojaStenflo2025,JarvenpaaMyllyojaSeuret2025+}, and in more general metric measure spaces \cite{EkstromJarvenpaas2020,Seuret2018,Myllyoja2024,FengJarvenpaasSuomala2018}---much less is known about the general version of the Dvoretzky covering problem. In fact, prior to this work, a complete solution was not known for any singular measure. One of the major obstacles in extending Kahane's approach to arbitrary measures is, in essence, that the condition analogous to $I_{\rrr}(\nu)=\infty$ for all measures $\nu$ on $C$, is difficult to use in the general case. In the classical setting, one does not in fact need to be concerned with all measures supported on $A$, but instead, there is a way to construct a ``characterising'' measure on $A$ whose divergence of the energy can be used to obtain a cover for $A$. The construction of this measure however depends in a crucial way on the properties of the Lebesgue measure on $\mathbb{T}$, namely, the translational invariance and the uniform distribution, and the construction breaks down for more general measures. We work around this problem with a simple but powerful application of the Jankov-von Neumann uniformisation theorem---a classical result in descriptive set theory guaranteeing the existence of a measurable choice function---and we believe that similar ideas could be useful also outside of this context. A similar application of the theorem also allows us to solve the covering problem for analytic instead of compact sets essentially for free. 

For random covering sets in higher dimensions, the Dvoretzky covering problem is open even for the Lebesgue measure on $\mathbb{T}^d$, although if one replaces the balls in the definition of the random covering set by homothetic simplices, then a potential theoretic characterisation is known again due to Kahane \cite{Kahane1990}. Unfortunately we are not able to handle the case of general radii for any measure on $\R^d$, but we can characterise the critical exponent for the covering problem for radii given by $\rrr=(cn^{-t})_n$. The proof of this characterisation is elementary and does not use the methods needed in the proof of the full characterisation in the case of the real line. At exactly the critical exponent, the covering property generally depends on the multiplicative constant $c$, as discussed in the classical setting.  While the critical exponent was already known before the present work in many cases, such as for all fully supported Borel probability measures on $\mathbb{T}$ \cite{Tang2012} and Gibbs measures on irreducible topological Markov chains \cite{Seuret2018}, the value for the critical constant was only known in the case of the Lebesgue measure on $\mathbb{T}$. To work out the critical constant in more general settings, one needs the full power of our main result, and as a non-trivial application of the result, in the second half of the article we characterise the critical constant for random covering sets driven by normalised Hausdorff measures on self-conformal sets in $\R$.

\section{Main Results}
\label{sec:main_results}
To state our main results precisely, we first introduce some relevant definitions and notation. We always denote the measure driving the random covering process by $\mu$. Since this measure is often clear from the context we will mainly use the shorter notation $\mathbb{P}$ for $\mathbb{P}_{\mu}$---the notation $\mathbb{P}_{\mu}$ is only used in the statements of our main results to clarify the dependence of the statement on $\mu$. We emphasise here that when we say that $E_{\rrr}$ has some property almost surely, we mean that $E_{\rrr}(\omega)$ has the property for $\mathbb{P}_{\mu}$-almost every $\omega$, that is the phrase \emph{almost surely} implicitly depends on the measure $\mu$ which changes depending on context. Throughout the article, we reserve the notation $X$ for the support of the measure $\mu$ and, again, the measure $X$ refers to should always be clear from the context.

For $A\subset \R^d$, we denote by $\PP(A)$ the collection of Borel probability measures whose support is contained in $A$ and by $\PPC(A)$ the measures in $\PP(A)$ with compact support. For a fixed probability measure $\mu\in\PP(\R)$ and a sequence $\rrr=(r_n)_n$, and for $\nu\in \mathcal{P}(A)$, we define the \emph{$(\mu,\underline{r})$-energy} of $\nu$ to be 
\begin{equation}\label{eq:defnofn-energy}
    I_{\mu,\underline{r}}(\nu)=\iint \exp\left(\sum_{n=1}^{\infty}\mu(B(x,r_n)\cap B(y,r_n))\right)\dd\nu (y)\dd\nu(x).
\end{equation}
For an analytic set $A\subset X$, we define the \emph{$(\mu, \underline{r})$-capacity} of $A$ by
\begin{equation}\label{eq:defncapacity}
    \capa_{\mu,\underline{r}}(A)=\sup \{I_{\mu,\underline{r}}(\nu)^{-1}\mid \nu \in \PPC(A) \},
\end{equation}
where we interpret $\infty^{-1}=0=\sup \emptyset$. We note that the $(\mu, \underline{r})$-capacity is monotone: if $B\subset A$, then $\capa_{\mu,\rrr}(B)\leq \capa_{\mu,\rrr}(A)$. Note also that $\capa_{\mu,\rrr}(A)=0$ if and only if $I_{\mu,\rrr}(\nu)=\infty$ for every $\nu \in \PPC(A)$. We also remark that
\begin{equation*}
    \mathcal{L}(B(x,r_n)\cap B(y,r_n))=\max\{2r_n-|x-y|,0\},
\end{equation*}
and the $\rrr$-capacity in Kahane's work, and in particular \cref{thm:kahane}, is precisely the $(\mathcal{L},\rrr)$-capacity in our language.

To state our main result, for a measure $\mu\in\PP(\R^d)$ and a sequence $\rrr=(r_n)_n$, we let
\begin{equation}\label{eq:finite-squares}
    X_{\mu,\rrr}=\left\{x\in X\colon \sum_{n=1}^{\infty}\mu(B(x,r_n))^2<\infty\right\}.
\end{equation}
The following generalisation of  is the main result of the paper, which solves the Dvoretzky covering problem for arbitrary Borel probability measures on $\R$.
\begin{theorem}\label{thm:main}
    Let $\mu\in\PP(\R)$, $\underline{r}=(r_k)_{k=1}^{\infty}$ be a non-increasing sequence of positive numbers tending to zero and let $A\subset \R$ be a non-empty analytic set. Then $A$ is $\mathbb{P}_{\mu}$-almost surely covered by $E_{\rrr}$ if and only if
    \begin{equation*}
        \capa_{\mu,\underline{r}}(A\cap X_{\mu,\rrr})=0.
    \end{equation*}
\end{theorem}
The set $X_{\mu,\rrr}$ is not visible in Kahane's \cref{thm:kahane}, because in the case of Lebesgue measure on $\mathbb{T}$, the set $X_{\mu,\rrr}$ is either the whole torus or it is the empty set. However, for general measures this is not the case, and the set $X_{\mu ,\rrr}$ can be a proper subset of the support of $\mu$. As will be seen in the proof of \cref{thm:main} (due to \cref{prop:1r-covering}), the complement of $X_{\mu, \rrr}$ is always automatically covered, which is why concidering measures supported on the part of $A$ inside of $X_{\mu,\rrr}$ is enough in \cref{thm:main}. 

Since the centres are chosen independently at random from the same distribution, and since the random covering process is invariant under reordering the sequence $\rrr$, the assumption that $\rrr$ is non-increasing does not result in loss of generality under the assumption that $r_k\to 0$. We also note that the case when $\rrr$ is bounded away from zero is not too interesting: it is easy to see that in this case the support of $\mu$ is fully covered by $E_{\rrr}$ almost surely. When the sequence $\rrr$ tends to zero, one of course cannot cover any point outside of the support of $\mu$ by $E_{\rrr}$, but we do not need to assume that $A\subset \spt\mu$ explicitly, since if there is a point $x\in A\setminus \spt\mu$, then $x\in X_{\mu,\rrr}$ and a simple calculation shows that $I_{\mu,\rrr}(\delta_x)<\infty$, where $\delta_x$ is the Dirac mass at $x$, so $\capa_{\mu,\rrr}(A\cap X_{\mu,\rrr})>0$.

As a particularly interesting special case, \cref{thm:main} solves the Dvoretzky covering problem for the support of $\mu$: the support $X$ is covered by $E_{\rrr}$ almost surely, if and only if $\capa_{\mu,\rrr}(X_{\mu,\rrr})=0$. Recalling that for the Lebesgue measure on $\mathbb{T}$, the condition $I_{\mathcal{L},\rrr}(\mathcal{L})=\infty$ characterises the covering problem for the full torus, one could hope to replace the condition $\capa_{\mu,\underline{r}}(X_{\mu,\rrr})=0$ with the condition $I_{\mu,\underline{r}}(\mu)=\infty$, at least when $X=X_{\mu,\rrr}$. Somewhat surprisingly, it turns out that that this cannot be done in general, even in cases where the measure $\mu$ is very regular. A counterexample is given, for instance, by the $s$-Hausdorff measure on the Cantor set, with $s=\frac{\log 2}{\log 3}$, where there are constants $c>0$, such that for $\rrr=(cn^{-\frac{1}{s}})$, we have $I_{\mu,\underline{r}}(\mu)=\infty$ but $\capa_{\mu,\underline{r}}(X)>0$. For more details and a discussion of the Dvoretzky covering problem for (normalised) Hausdorff measures on self-conformal fractals, see \cref{rem:cantor} and the end of this section.

The main difficulty in proving \cref{thm:main} is showing that vanishing capacity implies covering. The other direction is a straightforward modification of the following slightly weaker variant, which is well known, see for instance \cite{FanFriend2021}.
\begin{proposition}\label{prop:billard}
     Let $\mu\in\PP(\R^d)$, $\underline{r}=(r_k)_{k=1}^{\infty}$ be a non-increasing sequence of positive numbers tending to zero and let $A\subset \R$ be non-empty compact set. Assume that
     \begin{equation}\label{eq:sum-unif-finite}
         \sup_{x\in A}\sum_{n=1}^\infty\mu(B(x,r_n))^2<\infty.
     \end{equation}
      If $\capa_{\mu,\underline{r}}(A)>0$, then $A$ is $\mathbb{P}_{\mu}$-almost surely not covered by $E_{\underline{r}}$.
\end{proposition}

The proof of this proposition goes back to the ideas of Billard \cite{Billard1965} and Kahane \cite{Kahane1985}, and their approach forms the backbone to our methods as well. The approach, sometimes referred to as Billard's method, involves studying the second moments of the martingales
\begin{equation}\label{eq:martingale'}
    M_{k,\nu}(\omega)\coloneqq \int\prod_{n=1}^k\frac{1-\chi_{B(\omega_n,r_n)}(x)}{1-\mu(B(x,r_n))}\dd \nu(x),
\end{equation}
for $\nu\in\PP(A)$. One observes that if for some $\nu\in\PP(A)$ the martingale converges to a non-zero limit with positive probability, then there is a positive probability that some point in $A$ is not covered by the union $\bigcup_{n=1}^{\infty}B(\omega_n,r_n)$ and \emph{a fortiori}, not covered by $E_{\underline{r}}$ almost surely. 
As a technical note, we remark that whenever the quantity $M_{k,\nu}$ appears, we implicitly assume that $r_1$ is small enough so that 
\begin{equation}\label{eqn:measuresboundedawayfrom1}
    \sup_{x\in X}\mu(B(x,r_1))<1
\end{equation}
(of course, $M_{k,\nu}$ is not even well defined if, for example, $B(x,r_1)$ contains $X$ for $x\in X$). This is not a restriction to us, because forgetting finitely many terms from the beginning of a sequence $\rrr$ tending to zero does not affect either the random covering set $E_{\rrr}$ or the argument above (there is nothing special about the definition of $M_{k,\nu}$ starting from $1$, it could just as well be defined by starting from any other index). We prefer to start the indexing from $1$ to keep the notation less cumbersome, but the reader should be aware of this convention.

By the classical martingale convergence theorem, the martingale converges to a non-zero limit if the second moments $\mathbb{E}(M_{k,\nu}^2)$ are bounded, since then the martingale $(M_{k,\nu})$ converges in $L^2$ and hence in $L^1$, and since $\mathbb{E}(M_{k,\nu})=1$ for all $k$, the limit has expectation one and therefore is positive with positive probability. In particular, we have the following proposition.
\begin{proposition}[Billard's condition]\label{prop:billard2}
     Let $\mu\in\PP(\R^d)$, $\underline{r}=(r_k)_{k=1}^{\infty}$ be a non-increasing sequence of positive numbers tending to zero and let $A\subset \R$ be non-empty compact set. If there exists a measure $\nu\in\PP(A)$, such that $\mathbb{E}(M_{k,\nu}^2)$ is bounded, then $A$ is $\mathbb{P}_{\mu}$-almost surely not covered by $E_{\underline{r}}(\omega)$.
\end{proposition}
\cref{prop:billard} follows from \cref{prop:billard2} by using Fubini's theorem and Taylor's theorem for $\log(1+x)$, which shows that $1+x=\exp(x+O(x^2))$, together with \cref{eq:sum-unif-finite} to show that for every measure $\nu\in\PP(A)$,
\begin{align*}
    \mathbb{E}(M_{k,\nu}^2)&=\iint\prod_{n=1}^k\frac{1-\mu(B(x,r_n)-\mu(B(y,r_n)+\mu(B(x,r_n)\cap B(y,r_n))}{(1-\mu(B(x,r_n)))(1-\mu(B(y,r_n)))}\dd\nu(y)\dd\nu(x)\\
    &\approx \iint \exp\left(\sum_{n=1}^{\infty}\mu(B(x,r_n)\cap B(y,r_n))\right)\dd\nu (y)\dd\nu(x)=I_{\mu,\rrr}(\nu),
\end{align*}
for large values of $k$, at least when $\sup_{x\in A}\sum_{n=1}^{\infty}\mu(B(x,r_n))^2<\infty$. This calculation together with a uniformisation argument allows us to deduce \cref{thm:main} from the following result, which can be thought of as a probabilistic variant of our main theorem.
\begin{theorem}\label{thm:unbounded-moments}
    Let $\mu\in\PP(\R)$, $\underline{r}=(r_k)_{k=1}^{\infty}$ be a non-increasing sequence of positive numbers tending to zero and let $A\subset \R$ be a non-empty analytic set. Then $A$ is $\mathbb{P}_{\mu}$-almost surely covered by $E_{\underline r}$, if and only if $\mathbb{E}(M_{k,\nu}^2)$ is unbounded for all $\nu\in\PPC(A)$.
\end{theorem}
Let us briefly comment on our strategy for proving \cref{thm:unbounded-moments}. By \cref{prop:billard2} we only need to prove that the unboundedness of the second moments implies that $A$ is covered almost surely. The proof of this has two main difficulties. Firstly, we need a way for passing from information for all measures supported on $A$ to the covering property for $A$ itself. There are two ways one could hope to go about this. First possibility is to try to find a ``characterising'' measure on $X$, for which one could use the unboundedness of the second moments to show that all of $A$ is covered almost surely. As discussed in the introduction, this approach is viable for the Lebesgue measure, and it is essentially what Kahane uses in his proof of \cref{thm:kahane}. The construction of the measure is abstract and, in general, there is no easy way to see what this ``characterising'' measure is, except when one is trying to cover the whole torus $\mathbb{T}$, in which case it is the Lebesgue measure itself. The construction uses the rotational invariance as well as the uniform distribution of the Lebesgue measure in a crucial way, and completely breaks down in our general setting. Moreover, as discussed earlier, even for the random covering process driven by the Hausdorff measure on the Cantor set, the unboundedness of the second moments of the driving measure does not imply covering, which further rules out the viability of this approach.

Instead we use a more abstract approach enabled by the Jankov-von Neumann uniformisation theorem, which guarantees the existence of a (Borel) measurable choice function. In \cref{prop:ae-cover-suffices'}, we apply this result to show that it suffices to find for all measures $\nu\in\PPC(A)$, a \emph{deterministic} (that is independent of $\omega$) set $A_\nu\subset A$ of full $\nu$ measure which is covered by $E_{\rrr}$ almost surely. If this can be done, then \cref{prop:ae-cover-suffices'} implies that the entire set $A$ is covered almost surely. This point is very subtle, since in our general setting it follows from the Borel-Cantelli lemma that a given point $x\in X$ is covered almost surely if and only if
\begin{equation}\label{eq:measures-diverge}
    \sum_{n=1}^{\infty}\mu(B(x,r_n))=\infty,
\end{equation}
and for any $\nu\in\PP(A)$, Fubini's theorem then shows that if \cref{eq:measures-diverge} holds for all $x\in A$, then
\begin{align*}
    \mathbb{E}(\nu(E_{\rrr}))&=\iint\chi_{E_{\rrr}(\omega)}(x)\dd \nu(x)\dd\mathbb{P}(\omega)\\
    &=\iint\chi_{E_{\rrr}(\omega)}(x)\dd\mathbb{P}(\omega)\dd \nu(x)\\
    &=\int\mathbb{P}(x\in E_{\rrr})\dd \nu(x)=1,
\end{align*}
so $\nu(E_{\rrr})=1$ almost surely. In particular, already the condition \cref{eq:measures-diverge}, which is known to \emph{not} be a sufficient condition for full covering, shows that with respect to any measure on $A$, almost every point is covered almost surely. However, the full measure set which is covered almost surely is $E_{\rrr}(\omega)$ itself, which certainly depends on the realisation of the random process, and in order to find a deterministic set of full $\nu$-measure which is covered almost surely, one needs the stronger assumption that $\mathbb{E}(M_{k,\nu}^2)$ is unbounded. Our construction of such a set is inspired by the proof of Shepp's theorem in Kahane's book \cite{Kahane1985}, but requires substantial technical modifications. We should point out that \cref{prop:ae-cover-suffices'} works in $\R^d$ and enables a possible avenue for attacking higher dimensional versions of the problem, but the construction of the full measure subset which is covered almost surely, and in particular the crux of the argument, \cref{lemma:cruxlemma}, is the main obstruction in generalising \cref{thm:main} for measures on $\R^d$.

The final ingredient for obtaining  \cref{thm:main} from \cref{thm:unbounded-moments} is the following proposition, which shows that with an $\varepsilon$ improvement in the exponent in \cref{eq:measures-diverge}, the pointwise almost sure covering property can be upgraded to a full covering. In particular, this implies that the set $X\setminus X_{\mu,\rrr}$ is automatically covered almost surely.
\begin{proposition}\label{prop:1r-covering-prop}
    Let $\mu\in\PP(\R)$, $\varepsilon>0$ and let $A\subset \R$ be an analytic set with the property that $\sum_k\mu(B(x,r_k))^{1+\varepsilon}=\infty$ for every $x\in A$. Then 
    $A$ is $\mathbb{P}_{\mu}$-almost surely covered by $E_{\rrr}$.
\end{proposition}
The proof of \cref{prop:1r-covering-prop} is elementary and self-contained, and does not use any technical ideas needed in the proof of \cref{thm:unbounded-moments}. \cref{prop:1r-covering-prop} is quite powerful in itself, since a characterisation of the critical exponent for the Dvoretzky covering problem for polynomially decreasing radii for general measures follows very easily from it. In fact, while we were unable to prove \cref{prop:1r-covering-prop} for measures on $\R^d$, a slightly weaker variant, \cref{prop:1+epsilonk-covering}, which works in higher dimensions is enough for the application to the critical exponent to hold in $\R^d$. We denote by
\begin{equation*}
    \ldimloc(\mu,x)=\liminf_{r\to 0}\frac{\log\mu(B(x,r))}{\log r},
\end{equation*}
the \emph{lower local dimension} of $\mu$ at $x$.

\begin{theorem}\label{thm:seuret-generalization}
    Let $\mu\in\PP(\R^d)$, $c,t>0$ and let $\rrr=(cn^{-t})_n$.
    \begin{enumerate}
        \item If $\frac{1}{t}<\sup_{x\in X}\ldimloc(\mu,x)$, then $X$ is $\mathbb{P}_{\mu}$-almost surely not covered by $E_{\rrr}$.
        \item If $\frac{1}{t}>\sup_{x\in X}\ldimloc(\mu,x)$, then $X$ is $\mathbb{P}_{\mu}$-almost surely covered by $E_{\rrr}$.
    \end{enumerate}
\end{theorem}
We note that the first item in the theorem already follows very easily from the Borel-Cantelli lemma, but the second item requires \cref{prop:1r-covering-prop}. As mentioned in the introduction, prior to this work this result was known for all Borel probability measures with full supported on $\mathbb{T}$ \cite{Tang2012} and for Gibbs measures on irreducible topological Markov chains \cite{Seuret2018}, but the proofs in both of these articles are quite different from ours.

\subsection{Dvoretzky covering problem on self-conformal sets}
As it is already the case in the classical situation, when $t$ is equal to the critical exponent in \cref{thm:seuret-generalization}, the covering problem is very subtle and the covering property depends on the constant $c>0$. While we are unable to determine the critical constant for general measures on $\R$---and it is unclear to us that a satsifactory characterisation should exist in general in the first place---with a non-trivial application of \cref{thm:main} we can work out the critical constant for a large class of natural measures supported on fractal sets, namely, normalised Hausdorff-measures on strongly separated self-conformal sets.

To keep this section concise, we defer the precise definition of self-conformal sets to \cref{sec:self-conformal}---if the reader is not familiar with the definition, they are encouraged to keep in mind the simplest example of the classical middle-$\frac{1}{3}$ Cantor set. We denote by $\mathcal{H}^s$ the $s$ dimensional \emph{Hausdorff measure} or \emph{$s$-Hausdorff measure} for short. If a self-conformal set $X$ with Hausdorff dimension $s$ satisfies the strong separation condition, then it is well known that the $s$-Hausdorff measure of $X$ is positive and finite \cite{PeresEtAl2001}. Denoting by $\mu\coloneqq\mathcal{H}^s(X)^{-1}\mathcal{H}^s|_X$ the normalised $s$-Hausdorff measure on $X$, it is then standard that $\mu$ is \emph{$s$-Ahlfors regular}, that is, there is a constant $C>0$, such that
\begin{equation*}
    C^{-1}r^s\leq\mu(B(x,r))\leq  Cr^s,
\end{equation*}
for all $x\in X$ and $r>0$ and in particular, the critical exponent in \cref{thm:seuret-generalization} is $t=\frac{1}{s}$. It turns out that for this critical exponent, the critical constant for the covering problem with radii $(cn^{-\frac{1}{s}})$ depends in a subtle way on the multifractal structure of the average densities of $\mu$. Recall that the \emph{upper and lower average densities} of an $s$-Ahlfors regular measure $\mu$ at $x\in X$ are defined by 
\begin{equation*}
    \overline{\AA}(x)=\limsup_{t\to 0}\frac{1}{-\log t}\int_t^1\mu(B(x,r))r^{-s-1}\dd r,
\end{equation*}
and
\begin{equation*}
    \underline{\AA}(x)=\liminf_{t\to 0}\frac{1}{-\log t}\int_t^1\mu(B(x,r))r^{-s-1}\dd r,
\end{equation*}
respectively. If the limit exists at $x\in X$, we denote it by $\AA(x)$ and call it the \emph{average density} of $\mu$ at $x$. For $\alpha\in\R$ we denote by
\begin{equation*}
    f(\alpha)=\dimh\{x\in X\colon \AA(x)=\alpha\},
\end{equation*}
the \emph{multifractal spectrum} of the average density. The following result determines the critical constant for normalised $s$-Hausdorff measures on self-conformal sets.
\begin{theorem}\label{thm:self-conformal}
    Let $X$ be a strongly separated self-conformal set, let $\mu$ be the normalised $s$-Hausdorff measure on $X$ and let $\rrr=(cn^{-\frac{1}{s}})_n$.
    \begin{enumerate}
        \item\label{it:sc-no-cover} If $c<\left(\max_{\alpha}\frac{f(\alpha)}{s\alpha}\right)^{\frac{1}{s}}$, then $X$ is $\mathbb{P}_{\mu}$-almost surely not covered by $E_{\rrr}$.
        \item\label{it:sc-cover} If $c>\left(\max_{\alpha}\frac{f(\alpha)}{s\alpha}\right)^{\frac{1}{s}}$, then $X$ is $\mathbb{P}_{\mu}$-almost surely covered by $E_{\rrr}$. 
    \end{enumerate}
\end{theorem}
Unfortunately, our methods fall short for $c=\left(\max_{\alpha}\frac{f(\alpha)}{s\alpha}\right)^{\frac{1}{s}}$ and we suspect that the covering property at the critical constant depends on the underlying self-conformal set. For the proof of \cref{thm:self-conformal}, we show that the condition $\capa_{\mu,\rrr}(X)>0$ has a pointwise formulation in the self-conformal setting and the quantities involved in this pointwise version are naturally connected to the average densities of the $s$-conformal measure $\mu$ and the local dimensions of the reference measures $\nu$. We emphasise that \cref{thm:self-conformal} is new even in the simplest case when $X$ is the $\frac{1}{3}$-Cantor set and $\mu$ is the $s$-Hausdorff measure on $X$, with $s=\frac{\log2}{\log3}$.

It is well known that, for any self-conformal set, there is $\alpha_0>0$, such that $\AA(x)=\alpha_0$ for $\mu$-almost every $x\in X$ and while the quantity $\frac{f(\alpha)}{s\alpha}$ is in general difficult to analyse, a simple concavity argument shows that in the case of the Cantor set, the quantity is \emph{not} maximised at $\alpha=\alpha_0$. This is particularly interesting, because it implies that there are constants $c>0$, such that for sequences $\rrr=(cn^{-\frac{1}{s}})_n$,
\begin{equation*}
    I_{\mu,\rrr}(\mu)=\infty,
\end{equation*}
but nevertheless $X$ is not covered by $E_{\rrr}$ almost surely, see \cref{rem:cantor} for details.

As the final contribution of this article, we utilise basic tools from ergodic theory, to give numerical bounds for the critical constant in \cref{thm:self-conformal} for the Cantor set. In particular, for the lower bound for the critical constant, we are able to beat the trivial constant
\begin{equation}\label{eq:trivial-constant}
    \left(\frac{f(\alpha_0)}{s\alpha_0}\right)^{\frac{1}{s}}=\left(\frac{1}{\alpha_0}\right)^{\frac{1}{s}}= 1.057\ldots,
\end{equation}
given by the $\mu$-almost sure value of $\AA(x)$, see \cref{sec:cantor} or \cite[Section 6]{Falconer1997} for how to numerically compute $\alpha_0$. For the upper bound, the best we could do is the trivial upper bound
\begin{equation*}
    \left(\frac{f(\alpha)}{s\alpha}\right)^{\frac{1}{s}}\leq\left(\frac{s}{s\alpha_{\min}}\right)^{\frac{1}{s}}=  \left(\frac{1}{\alpha_{\min}}\right)^{\frac{1}{s}},
\end{equation*}
where $\alpha_{\min}=\inf_{x\in X}\AA(x)$. It is possible to show that $\AA(x)$ is minimised at the endpoints of the construction intervals and the value of $\AA(0)$ can again be estimated numerically. The precise result with concrete constants which we prove in \cref{sec:cantor} is as follows.
\begin{proposition}\label{prop:cantor}
    Let $\mu$ be the normalised $s$-Hausdorff measure on the $\frac{1}{3}$-Cantor set $X$ and let $\rrr=(cn^{-\frac{1}{s}})_n$.
    \begin{enumerate}
        \item\label{it:concrete1} If $c<1.06126$, then $X$ is $\mathbb{P}_{\mu}$-almost surely not covered by $E_{\rrr}$.
        \item\label{it:concrete2} If $c>1.37546$, then $X$ is $\mathbb{P}_{\mu}$-almost surely covered by $E_{\rrr}$.
    \end{enumerate}
\end{proposition}
\cref{prop:cantor} gives the first explicit constants for both covering and not covering the Cantor set for the sequence $\rrr=(cn^{-\frac{1}{s}})_n$. Both of the constants are computed numerically with a computer assisted calculation. The constant in \cref{it:concrete2} is obtained by treating every point in $X$ as if the average density was equal to $\alpha_{\min}$ and is probably far from optimal, since in reality $f(\alpha_{\min})=0$. The improvement we were able to make in \cref{it:concrete1} from the trivial constant in \cref{eq:trivial-constant} is quite small, so we expect the actual critical constant to be somewhat close to the value in \cref{it:concrete1}. Finding ways to numerically calculate the critical constant, or even ways to give non-trivial upper bounds, would be of interest. We note that it follows from the proofs that, if instead of covering the entire Cantor set we wanted to cover the Borel set of full $\mu$-measure, where the average density equals the almost sure value, then the critical constant is given by \eqref{eq:trivial-constant}. 

\subsection{Structure of the article}
The majority of the technical work in the paper is done in \cref{sec:capacity_and_coverings_of_analytic_sets}, where we prove the probabilistic formulation of our main result, \cref{thm:unbounded-moments}. 
In \cref{sec:proofofthmmain} we show the equivalence between \cref{thm:unbounded-moments} and \cref{thm:main}.
In \cref{sec:criticalexponent} we prove a weaker version of \cref{prop:1r-covering-prop} in $\R^d$ and use it to prove \cref{thm:seuret-generalization}. In the final two sections, \cref{sec:self-conformal,sec:cantor}, we study the Dvoretzky covering problem on self-conformal sets and prove \cref{thm:self-conformal,prop:cantor}.

\section{Proof of Theorem \ref{thm:unbounded-moments}}\label{sec:capacity_and_coverings_of_analytic_sets}
In this section, we prove \cref{thm:unbounded-moments}. Let us start by setting up some notation. Throughout this section, we let $U_k(\omega)=\bigcup_{n=1}^{k}B(\omega_n,r_n)$ and $F_k(\omega)=\R\setminus U_k(\omega)$. 
Going forward, we assume that $r_1>0$ is small enough so that the condition \eqref{eqn:measuresboundedawayfrom1} is satisfied.
This is mainly to ensure that $\mu(X\setminus B(x,r_n))=1-\mu(B(x,r_n))>0$, for every $x\in X$ and every $n\in \N$; the uniform bound for measures of balls away from $1$ is only needed in Lemma \ref{lemma:secondmomentcomparabletoexp}. In particular, with this assumption,
\begin{equation*}
    \mathbb{P}(x\in F_k)=\prod_{n=1}^{k}(1-\mu(B(x,r_n)))>0,
\end{equation*}
for all $k\in\N$. To simplify our notation slightly, we introduce the following notation. For $x,y\in X$ and $k\in\N$, we write
\begin{align}\label{eq:p_k-def}
    p_k(x,y)&\coloneqq\frac{\mathbb{P}(y\in F_k\;|\;x\in F_k)}{\mathbb{P}(y\in F_k)}=\frac{\mathbb{P}(y,x\in F_k)}{\mathbb{P}(y\in F_k)\mathbb{P}(x\in F_k)}=\frac{\mathbb{E}(\chi_{F_k}(x)\chi_{F_k}(y))}{\mathbb{P}(y\in F_k)\mathbb{P}(x\in F_k)}\\
    &=\prod_{n=1}^k\frac{1-\mu(B(x,r_n)-\mu(B(y,r_n)+\mu(B(x,r_n)\cap B(y,r_n))}{(1-\mu(B(x,r_n)))(1-\mu(B(y,r_n)))}\nonumber.
\end{align}
For $\mu\in\PP(\R)$, $\rrr=(r_n)_n$ and $\nu\in \PP(X)$, we let 
\begin{equation}\label{eq:defnofn-energy'}
    J_{\mu,\rrr,k}(\nu)=\int \int p_k(x,y)\, d\nu (y)\, d\nu(x).
\end{equation}
It is easy to see from \cref{eq:p_k-def,eq:defnofn-energy'} using Fubini, that $J_{\mu,\rrr,k}(\nu)$ is nothing but the second moment of the martingale in \cref{eq:martingale'}. This implies that for any $\nu$, the sequence $(J_{\mu,\rrr,k}(\nu))_{k=1}^{\infty}$ is increasing and thus has a (possibly infinite) limit. Let us write
\begin{equation}\label{eq:defnenergy}
J_{\mu,\rrr}(\nu)\coloneqq\lim_{k\to \infty}  J_{\mu,\rrr,k}(\nu)=\lim_{k\to\infty}\mathbb{E}(M_{k,\nu}^2).
\end{equation}
In this section, we prove the following result, which is a restatement of \cref{thm:unbounded-moments}.
\begin{theorem}\label{thm:analyticsetcoverediffcapacityzero}
    Let $\mu\in\PP(\R)$, $\underline{r}=(r_k)_{k=1}^{\infty}$ be a non-increasing sequence of positive numbers tending to zero and let $A\subset \R$ be an analytic set. Then 
    \begin{equation*}
        A\subset E_{\underline r},
    \end{equation*}
    $\mathbb{P}_{\mu}$-almost surely, if $J_{\mu,\rrr}(\nu)=\infty$ for all $\nu\in\PPC(A)$. 
\end{theorem}

For the remainder of the section we consider $\mu\in\PP(\R)$ and $\rrr=(r_n)_n$ fixed, and write $X\coloneqq \spt\mu$ and $\mathbb{P}=\mathbb{P}_{\mu}$. We also omit $\mu$ and $\rrr$ from the notation and write $J_k(\nu)=J_{\mu,\rrr,k}(\nu)$ and $J(\nu)=J_{\mu,\rrr}(\nu)$. Let us also fix an analytic set $A\subset \R$. Note that for any point $a\in A\setminus X$, we have $J(\delta_a)<\infty$ and indeed $\mathbb{P}(a\in E_{\rrr})=0$, so we may assume that $A\subset X$.

We start working towards the proof of \cref{thm:analyticsetcoverediffcapacityzero} by making a simple reduction. Observe that 
\begin{align*}
X=&\{x\in X\mid \mu(\partial B(x,r_k))=0 \text{ for every } k\in \N \}\cup \left(\bigcup_{k\in \N}\{x\in X\mid \mu(\partial B(x,r_k))>0\}
\right)\\
\eqqcolon&\, X_{0}\cup X_{atomish}.
\end{align*}
Since the boundary of any ball intersects $X$ at at most two points (and the set of atoms of $\mu$ is countable), $X_{atomish}$ is a countable set. By applying the condition $J(\delta_a)=\infty$ to the Dirac measures $\delta_a$ for $a\in A$, one obtains via a simple calculation that $\sum_{k}\mu(B(a,r_k))=\infty$ for every $a\in A$. Thus if $J(\delta_a)=\infty$ for all $a\in X_{atomish}$, then the countable set $A\cap X_{atomish}$ is covered almost surely. Hence under the assumption of \cref{thm:analyticsetcoverediffcapacityzero}, $A$ is covered almost surely if and only if $A\cap X_{0}$ is covered almost surely (of course, if $A\cap X_0=\emptyset$, this is always true and this trivial case is already proved).

The following two propositions, inspired by \cite[Proposition 2.11]{JarvenpaaMyllyojaSeuret2025+}, allow us to bridge the gap between having information on all measures supported on $A$ to obtaining information on the probability of covering the whole set $A$. While the proofs of the propositions are very simple applications of the Jankov-von Neumann uniformisation theorem, the results are essential in the proof of \cref{thm:analyticsetcoverediffcapacityzero}.

\begin{proposition}\label{prop:ae-cover-suffices'}
    Suppose that for every $\nu\in\PP(X)$ there exists a Borel set $X_{\nu}\subset X$, with $\nu(X_{\nu})=1$, such that $\mathbb{P}(X_{\nu}\subset E_{\rrr})=1$. Then
    \begin{equation*}
        \mathbb{P}(X\subset E_{\rrr})=1.
    \end{equation*}
\end{proposition}
\begin{proof}
    Assume to the contrary that $\mathbb{P}(X\setminus E_{\rrr}\ne \emptyset)>0$, which readily implies that $\mathbb{P}(X\setminus E_{\rrr}\ne \emptyset)=1$, since $X\setminus E_{\rrr}\ne \emptyset$ is a tail event. Let
    \begin{equation}\label{eq:G-def}
        G=\{(\omega,x)\in\Omega\times X\colon x\not\in E_{\rrr}(\omega)\},
    \end{equation}
    which is an analytic set. Let $\pi_1\colon \Omega\times X\to \Omega$ denote the projection onto the first coordinate. By the Jankov-von Neumann unformization theorem, see e.g. \cite[Theorem 18.1]{Kechris1995}, there exists a uniformising function $f\colon \pi_1(G)\to X$, which is measurable with respect to the $\sigma$-algebra generated by analytic sets and which satisfies $f(\omega)\in G_{\omega}\coloneqq \{x\in X\colon x\not\in E_{\rrr}(\omega)\}=X\setminus E_{\rrr}(\omega)$, for all $\omega\in\pi_1(G)$. Thus we may define $\nu\coloneqq f_*\mathbb{P}$, which is a Borel measure with
    \begin{equation*}
        \nu(X)=\mathbb{P}(f^{-1}(X))=\mathbb{P}(\pi_1(G))=1.
    \end{equation*}
    Now if $C\subset X$ is any Borel set with $\nu(C)=1$, we have
    \begin{align*}
        1=\nu(C)=\mathbb{P}(\{\omega\colon f(\omega)\in C\cap (X\setminus E_{\rrr}(\omega))\})\leq \mathbb{P}(C\setminus E_{\rrr}\ne\emptyset),
    \end{align*}
    which is a contradiction.
\end{proof}

\begin{proposition}\label{prop:analyticcoveredifeverycompactcovered}
    Suppose that $A\subset X$ is an analytic set which has the property that $\mathbb{P}(C\subset E_{\underline{r}})=1$ for every compact $C\subset A$.   
    Then 
    \begin{equation*}
        \mathbb{P}(A\subset E_{\underline{r}})=1.
    \end{equation*}
    
\end{proposition}
\begin{proof}
    Again, assume towards contradiction that $\mathbb{P}(A\setminus E_{\rrr}\ne \emptyset)=1$. Let $G^A=(\Omega\times A)\cap G$, where $G$ is as in \cref{eq:G-def}. Note that $G^A$ is the intersection of two analytic sets, hence also analytic. Again, using Jankov-von Neumann uniformization theorem, there exists $f\colon \pi_1(G^A)\to A$, which satisfies $f(\omega)\in G^A_{\omega}\coloneqq\{x\in A\colon x\not\in E_{\rrr}(\omega)\}=A\setminus E_{\rrr}(\omega)$, for each $\omega\in\pi_1(G^A)$. Again, we define $\nu\coloneqq f_*\mathbb{P}$ which is a Borel probability measure on $X$ with
    \begin{equation*}
        \nu(A)=\mathbb{P}(f^{-1}(A))=\mathbb{P}(\pi_1(G))=1.
    \end{equation*}
    Since $\nu$ is inner regular, as a Borel measure on $\R$, and $A$ is analytic, there exists a compact set $C\subset A$ with $\nu(C)>0$. Then
    \begin{align*}
        0<\nu(C)=\mathbb{P}(\{\omega\colon f(\omega)\in C\setminus E_{\rrr}\})\leq \mathbb{P}(C\setminus E_{\rrr}(\omega)\ne\emptyset),
    \end{align*}
    which is a contradiction.
\end{proof}

By the previous propositions it now suffices to consider compact subsets of $C\subset A\cap X_0$, and for each probability measure supported on $C$, to find a deterministic Borel set (i.e. a set independent of $\omega$) of full measure which is covered almost surely. To build this subset for a fixed measure $\nu$, we utilise the assumption that $J(\nu)=\infty$. We start with the following lemma.

\begin{lemma}\label{lemma:ifcapacityzerothenmeasureofgoodtendsto1}
   Suppose that $A\subset X$ is an analytic set which satisfies $J(\nu)=\infty$ for all $\nu\in\PPC(A)$. Let $C\subset A\cap X_0$ be compact. Then for every $\nu\in \PP(C)$ and for every $M>0$, 
    \begin{equation}\label{eq:divergtoinftyinprobability}
        \lim_{k\to \infty}\nu\left(\left\{x\in C\colon \int p_k(x,y)\, d\nu (y)\geq M\right\}\right)=1.
    \end{equation}   
\end{lemma}

\begin{proof}
If $C=\emptyset$ the claim is vacuous so assume this is not so. Since $\PP(C)\subset \PPC(A)$, we have that $J(\nu)=\infty$ for every $\nu \in \PP(C)$.

Let $\nu \in \PP(C)$ and write $\phi_k(x)=\int p_k(x,y)\,d\nu(y)$.
Assume that the claim is not true for $\nu$. Then there exist $0<\varepsilon,M<\infty$ and a strictly increasing sequence $(k_l)$ of integers such that for every $l\in \N$,
\begin{equation}\label{eq:setofpointswsmallpotential}
 \nu(B_{M,k_l})>\varepsilon,
\end{equation}
where 
\begin{equation}\label{eq:potentiallessthanM}
  B_{M,k}=\{x\in C\colon \phi_k(x)< M \}.
\end{equation}
Write $\nu_l=\nu_{B_{M,k_l}}$. By \eqref{eq:setofpointswsmallpotential} and \eqref{eq:potentiallessthanM}, for every $l\in \N$, 
\begin{align}\label{eq:energyofconditionalmeasurebounded}
    J_{k_l}(\nu_l)&=\nu(B_{M,k_l})^{-2}\int_{B_{M,k_l}} \int_{B_{M,k_l}}p_{k_l}(x,y)\,d\nu(y)\,d\nu(x) \nonumber\\
    &\leq \nu(B_{M,k_l})^{-2}\int_{B_{M,k_l}} \phi_{k_l}(x)\,d\nu(x)\leq \nu(B_{M,k_l})^{-1}M\leq \frac{M}{\varepsilon}.
\end{align}
Since $\nu_l\in \mathcal{P}(C)$ for every $l$, there is a subsequence $(\nu_{l_j})_j$ and $\eta \in \mathcal{P}(C)$ such that $\nu_{l_j}\to \eta$ in the weak-$\ast$ topology as $j\to \infty$. By continuity of the map $\PP(C)\to \PP(C\times C)\colon \sigma\mapsto \sigma\times\sigma$, we have that
$\nu_{l_j}\times \nu_{l_j}\to \eta\times \eta$ in weak-$\ast$.

Observe now that since $\mu(\partial B(z,r_n))=0$ for every $n\in \N$ and $z\in C$, and
\begin{equation*}
    p_k(x,y)=\prod_{n=1}^k\frac{1-\mu(B(x,r_n)-\mu(B(y,r_n)+\mu(B(x,r_n)\cap B(y,r_n))}{(1-\mu(B(x,r_n)))(1-\mu(B(y,r_n)))},
\end{equation*}
the map $(x,y)\mapsto p_k(x,y)$ is bounded and continuous on $C\times C$ for every $k$.
Thus, for every $k\in \N$, $J_k(\eta)=\lim_{j\to \infty} J_k(\nu_{l_j})$.

Recalling that the sequence $(J_k(\sigma))_{k=1}^{\infty}$ is increasing for any $\sigma \in \PP (C)$, we obtain that for every $k$, 
$J_k(\nu_{l_j})\leq I_{k_{l_j}}(\nu_{l_j})$ for all $j$ large enough so that $k_{l_j}\geq k$. Thus by \eqref{eq:energyofconditionalmeasurebounded}, for every $k\in \N$,
\begin{align*}
    J_k(\eta)&=\lim_{j\to \infty} J_k(\nu_{l_j})\leq \liminf_{j\to \infty} J_{k_{l_j}}(\nu_{l_j})\leq \frac{M}{\varepsilon}.
\end{align*}
This implies that $J(\eta)\leq \frac{M}{\varepsilon}$, which is a contradiction.
 
\end{proof}
Our next objective is to show that the condition \eqref{eq:divergtoinftyinprobability} implies the existence of a deterministic set, which is covered almost surely and has full $\nu$-measure. The following simple lemma is a crucial ingredient in the proof, and it is the component of the proof that remains essentially unchanged from the previously known setting of uniformly  distributed centres; see e.g. the lemma on page 146 of \cite{Kahane1985}.
\begin{lemma}[Crux lemma]\label{lemma:cruxlemma}
    Let $x,y\in \R$, $x<y$, and let $A\subset \R$ be a $\mu$-measurable set. The following properties hold.
    \begin{enumerate}
        \item[i)] If $A\subset (-\infty,x)$ and $\mathbb{P}(A\subset U_k,\, x\in F_k)>0$, then
        \[
        \mathbb{P}(y\in F_k\;|\; A\subset U_k,\, x\in F_k)\geq \mathbb{P}(y\in F_k\;|\;  x\in F_k).
        \]
        \item[ii)] If $A\subset (y,\infty)$ and  $\mathbb{P}(A\subset U_k,\, y\in F_k)>0$, then
        \[
        \mathbb{P}(x\in F_k\;|\; A\subset U_k,\, y\in F_k)\geq \mathbb{P}(x\in F_k\;|\;  y\in F_k).
        \]
    \end{enumerate}
\end{lemma}

\begin{proof}
    We will only prove $i)$, but the proof of $ii)$ is identical. 
    First note that the inequality in $i)$ is equivalent to
    \begin{equation*}
         \mathbb{P}(y\in U_k\;|\; A\subset U_k,\, x\in F_k)\leq \mathbb{P}(y\in U_k\;|\;  x\in F_k).
    \end{equation*}
    If the right hand side is zero, then so is the left and the conclusion of the lemma holds. Otherwise, we divide both sides by $\mathbb{P}(y\in U_k,\, x\in F_k)$ and multiply by $\mathbb{P}(A\subset U_k,\, x\in F_k)$ to obtain the equivalent inequality
    \begin{align*}
         \mathbb{P}(A\subset U_k\;|\; y\in U_k,\, x\in F_k)&\leq \mathbb{P}(A\subset U_k\;|\;  x\in F_k)\\
         &=\mathbb{P}(y\in F_k\;|\;  x\in F_k)\,\mathbb{P}(A\subset U_k\;|\; y\in F_k,\, x\in F_k)\\
         &+(1-\mathbb{P}(y\in F_k\;|\;  x\in F_k))\,\mathbb{P}(A\subset U_k\;|\; y\in U_k,\, x\in F_k),
    \end{align*}
    which in turn is equivalent to
    \begin{equation}\label{eq:crux100}
         \mathbb{P}(A\subset U_k\;|\; y\in U_k,\, x\in F_k)\leq \mathbb{P}(A\subset U_k\;|\; y\in F_k,\,  x\in F_k).
    \end{equation}
    We now establish some notation in order to study the left hand side. For a set $\Lambda\subset \{1,\ldots,k\}\eqqcolon\N_{\leq k}$, define the sets $U_{\Lambda}=\bigcup_{n\in \Lambda}B(\omega_n,r_n)$, $I_{\Lambda}=\bigcap_{n\in \Lambda}B(\omega_n,r_n)$ and $F_{\Lambda}=\R\setminus U_{\Lambda}$. Also, given $\Lambda$, write $\Delta=\N_{\leq k}\setminus \Lambda$ for the complement of $\Lambda$.
    We now make the observation that the event $\{x\in F_k,\, y\in U_k \}$ can be written as a disjoint union
    \[
    \{x\in F_k,\,y\in U_k \}=\bigcup_{\substack{\Lambda\subset \N_{\leq k},\\ \Lambda\neq \emptyset}}\{ x\in F_k,\, y\in I_{\Lambda}\cap F_{\Delta}\}\eqqcolon \bigcup_{\substack{\Lambda\subset \N_{\leq k},\\ \Lambda\neq \emptyset}} O_{\Lambda}.
    \]
    Thus the left hand side of \eqref{eq:crux100} can be written as
    \begin{equation}\label{eq:curx101}
      \mathbb{P}(A\subset U_k\;|\; y\in U_k,\, x\in F_k)=\sum_{\substack{\Lambda\neq \emptyset,\\ \mathbb{P}(O_{\Lambda})>0}}\frac{\mathbb{P}(O_{\Lambda})}{\mathbb{P}(x\in F_k,\, y\in U_k)}\, \mathbb{P}(A\subset U_k\;| \; O_{\Lambda}).  
    \end{equation}
    Fix now $\Lambda \subset \N_{\leq k}$ such that $\mathbb{P}(O_{\Lambda})>0$. Observe that given the event $O_{\Lambda}$, we have that $x\in F_k$ and $y\in I_{\Lambda}$, hence $A\cap U_{\Lambda}=\emptyset$ by the assumptions $A\subset (-\infty,x)$ and $y>x$. Observing also that $F_k=F_{\Lambda}\cap F_{\Delta}$, and that the events appearing in the calculation below indexed by $\Lambda$ are independent of those indexed by  $\Delta$, we obtain
    \begin{align*}
       \mathbb{P}(A\subset U_k\;| \; O_{\Lambda})&=
       \mathbb{P}(A\subset U_{\Delta}\; |\; x\in F_{\Lambda},\, x\in F_{\Delta},\, y\in I_{\Lambda},\, y\in F_{\Delta})\\
       &=\mathbb{P}(A\subset U_{\Delta}\; |\; x\in F_{\Delta},\,  y\in F_{\Delta})\\
       &= \mathbb{P}(A\subset U_{\Delta}\; |\; x\in F_{k},\,  y\in F_{k})\\
       &\leq \mathbb{P}(A\subset U_{k}\; |\; x\in F_{k},\,  y\in F_{k}).
    \end{align*}
    Plugging this estimate into \eqref{eq:curx101} and summing over $\Lambda$ establishes \eqref{eq:crux100} and the lemma is proved.
  \end{proof}

The next proposition is the main technical result needed in the proof of \cref{thm:unbounded-moments}.
\begin{proposition}\label{prop:no-billard-cover'}
 Let $C\subset A\cap X_0$ be a nonempty compact set.   Let $\nu\in\PP(C)$ and assume that 
    \begin{equation*}
        \lim_{k\to \infty}\nu\left(\left\{x\in C\colon \int p_k(x,y)\, d\nu (y)\geq M\right\}\right)=1
    \end{equation*}
    for every $M>0$. Then there exists a Borel set $C_{\nu}\subset C$, with $\nu(C_{\nu})=1$, such that $\mathbb{P}(C_{\nu}\subset E_{\rrr})=1$.
\end{proposition}

\begin{proof}
Write $a=\inf C$ and $b=\sup C$.
For every $M,k\in \N$, define the Borel sets
\begin{align*}
    &Y_{M,k}^-=\left\{x\in C\colon \int_{[a,x]}p_k(x,y)\dd\nu(y)\geq M\right\},\\
    &Y_{M,k}^+=\left\{x\in C\colon \int_{[x,b]} p_k(x,y)\dd\nu(y)\geq M\right\},
\end{align*}
and write $Y_{M,k}=Y_{M,k}^+\cup Y_{M,k}^-$. Our assumption implies that for every $M$,  
\begin{equation}\label{eq:ymkhasfullmeasureasktoinfty'}
\lim_{k\to \infty}\nu(Y_{M,k})=1.
\end{equation}
Recall that $F_k(\omega)=\R\setminus U_k(\omega)$, where $U_k(\omega)=\bigcup_{n=1}^kB(\omega_n,r_n)$. Our aim is to show that for every $M,k\in \N$,
\begin{equation}\label{eq:ymk+-covered'}
    \mathbb{P}(Y\cap F_k\neq \emptyset)\leq \frac{1}{M},
\end{equation}
for both $Y=Y_{M,k}^-$ and $Y=Y_{M,k}^+$.

Let us first explain how \eqref{eq:ymk+-covered'} implies the claim of the proposition. First observe that by \eqref{eq:ymk+-covered'}, for every $M,k\in \N$,
\[
\mathbb{P}(Y_{M,k}\cap F_k\neq \emptyset)=\mathbb{P}((Y_{M,k}^+\cap F_k\neq \emptyset)\cup(Y_{M,k}^-\cap F_k\neq \emptyset))\leq \frac{2}{M}.
\]
Let $(\varepsilon_n)_n$ be a sequence of positive numbers such that $\sum_n\varepsilon_n<1$. 
By \eqref{eq:ymkhasfullmeasureasktoinfty'}, for every $M\in \N$, it is possible to choose $k_M\in \N$ such that $\nu(Y_{M,k_M})>1-\varepsilon_M$. For each $n\in \N$, let
\[
W_n\coloneqq \bigcap_{M=n}^{\infty}Y_{M,k_M}.
\]
Then $\nu(W_1)>0$ and $\nu(W_n)\to 1$ as $n\to \infty$.
Observe now that for every $n\in \N$ and $M\geq n$,
\begin{align*}
    \mathbb{P}\left( W_n\subset U_{\infty} \right)
    &\geq \mathbb{P}\left( Y_{M,k_M}\subset U_{\infty} \right)
    \geq \mathbb{P}\left( Y_{M,k_M}\subset U_{k_M} \right)\geq 1-\frac{2}{M}.
\end{align*}
This implies that $ \mathbb{P}\left( W_n\subset U_{\infty} \right)=1$ for every $n$ and thus $C_1\coloneqq\bigcup_{n=1}^{\infty}W_n$ is a set with full $\nu$-measure which is covered by $U_{\infty}=\bigcup_{j=1}^{\infty}B(\omega_j,r_j)$ almost surely. We then repeat the argument for each $l\in \N$ to find a set $C_l$ of full $\nu$-measure which is covered by $\bigcup_{j=l}^{\infty}B(\omega_j,r_j)$ almost surely. Here we need that the condition in the statement does not depend on starting from the first index in the definition of the martingale. The set $C_{\nu}\coloneqq\bigcap_lC_l$ is then a set with full $\nu$-measure which is almost surely covered by $E_{\underline{r}}(\omega)$.

It thus remains to establish \eqref{eq:ymk+-covered'}. We do this only for $Y_{M,k}^+$ as the argument for $Y_{M,k}^-$ is symmetric. To keep our notation from getting too crowded, we write $Y=Y_{M,k}^+$  for the remainder of the proof.

Recall that by our assumption, 
\begin{equation}\label{eq:uniform-integral-bound'}
    \int_{[x,b]}p_k(x,y)\dd\nu(y)\geq M,
\end{equation}
for all $x\in Y$. Let $A_k$ denote the event $\{\omega\colon F_k(\omega)\cap Y\ne\emptyset\}$ and let $\mathcal{D}_N(Y)=\{I\colon I\in\mathcal{D}_N,\, I\cap Y\ne\emptyset\}$, where $\mathcal{D}_N$ is the collection of level $N$ dyadic intervals. We claim that
\begin{equation}\label{eq:prob-convergence'}
    \mathbb{P}(A_k)=\lim_{N\to\infty}\mathbb{P}(A_{k,N}),
\end{equation}
where $A_{k,N}=\{\omega\colon\exists I\in \mathcal{D}_N(Y)\text{ such that }I\subset F_k(\omega)\}$. For this, it suffices to show that $\mathbb{P}(Y\cap F_k\ne\emptyset)=\mathbb{P}(Y\cap \mathrm{int}(F_k)\ne\emptyset)$. Observe first that, regardless of $\omega$, the set $[a,b]\cap F_k(\omega)=[a,b]\setminus U_k(\omega)$ is a finite union of isolated points and closed intervals with positive length. Furthermore, with probability $1$, the set of isolated points of $[a,b]\cap F_k(\omega)$ is contained in $X_{atomish}$, hence the assumption $Y\subset C\subset X_0$ implies that $\mathbb{P}(Y\cap F_k\ne\emptyset)=\mathbb{P}(Y\cap F'_k\ne\emptyset)$, where $F'_k(\omega)\subset F_k(\omega)$ is a finite union of closed intervals with positive length. Let
\begin{equation*}
Y'\coloneqq  \{y\in Y\colon [y-r_y,y)\cap Y=\emptyset \text{ or } (y,y+r_y]\cap Y=\emptyset \text{ for some } r_y>0\}.
\end{equation*}
It is a simple exercise to show that $Y'\subset Y$ is a countable set and since $\mathbb{P}(y\in \partial B(\omega_n,r_n) \text{ for some } n\in \N)=0$ for every $y\in Y$ (since $Y\subset X_0$), we have that 
$\mathbb{P}(Y'\cap \partial B(\omega_n,r_n)\neq \emptyset \text{ for some } n=1,\ldots , k)=0.$
Thus 
$\mathbb{P}(Y\cap F_k\ne\emptyset)=\mathbb{P}(W)$, where
\[
W\coloneqq \{\omega \colon F_k'(\omega)\cap Y\neq \emptyset, Y'\cap \partial B(\omega_n,r_n)=\emptyset \text{ for all } n=1,\ldots,k \}.
\]
Note now that if $\omega\in W$, then there exists $y_{\omega}\in Y\cap F'_k(\omega)$. If $y_{\omega}\in Y'$, then $y_{\omega}\in [a,b]\setminus (\bigcup_{n=1}^k\overline{B}(\omega_n,r_n))$, hence $y_{\omega}\in \mathrm{int}(F_k(\omega))\cap Y$. If $y_{\omega}\in (Y\setminus Y')\cap F'_k(\omega)$, then $y_{\omega}\in I(\omega)$, where $I(\omega)\subset F_k(\omega)$ is a closed interval. If $y_{\omega}\in \mathrm{int}(I(\omega))$, we are finished, otherwise $y_{\omega}$ is one of the endpoints of $I(\omega)$, say the right one, in which case, since $[y_{\omega}-r,y_{\omega})\cap Y\ne\emptyset$ for all $r>0$, there is a point $z\in Y\cap \mathrm{int}(I(\omega))$. This finishes the proof of \eqref{eq:prob-convergence'}.

Now for any $N\in\N$, we denote by $I_1^N,\ldots,I_{\#\mathcal{D}_N(Y)}^N$, the enumeration of $\mathcal{D}_N(Y)$ from left to right, and fix arbitrary points $x_i^N\in I_i^N\cap Y$, for all $i=1,\ldots,\#\mathcal{D}_N(Y)$. We let
\begin{equation*}
    A_{k,N,j}=\{\omega\colon x_{\ell}^N\in U_k(\omega)\,\forall \ell<j\text{ and }x_j^N\in F_k(\omega)\}.
\end{equation*}
These sets are disjoint for distinct indices $j$ and clearly $A_{k,N}\subset \bigcup_{j}A_{k,N,j}$. 

Observe now that that if $\mathbb{P}(A_{k,N,j})>0$ and $x>x_j^N$, then part $i)$ of \cref{lemma:cruxlemma} implies that
\begin{equation}\label{eq:implicationofcrux}
    \mathbb{P}(x\in F_k\:|\; A_{k,N,j})\geq  \mathbb{P}(x\in F_k\:|\; x_j^N\in F_k)
\end{equation}
(note also that \eqref{eq:implicationofcrux} trivially holds as an equality if $x=x_j^N$).
Now by Fubini's theorem, \cref{eq:implicationofcrux} and \eqref{eq:uniform-integral-bound'}, we get
    \begin{align*}\label{eq:after-crux'}
        1&=\mathbb{E}\left(\int \frac{\chi_{F_k}(x)}{\mathbb{P}(x\in F_k)}\dd\nu(x)\right)\geq \mathbb{E}\left(\chi_{\bigcup_{j}A_{k,N,j}}\int \frac{\chi_{F_k}(x)}{\mathbb{P}(x\in F_k)}\dd\nu(x)\right)\\
        &= \sum_j\mathbb{E}\left(\chi_{A_{k,N,j}}\int \frac{\chi_{F_k}(x)}{\mathbb{P}(x\in F_k)}\dd\nu(x)\right)\nonumber\\
         &\geq \sum_j\mathbb{P}(A_{k,N,j})\int_{[x_j^N,b]} \frac{\mathbb{P}(x\in F_k\:|\; A_{k,N,j})}{\mathbb{P}(x\in F_k)}\dd\nu(x)\nonumber\\
        &\geq \sum_j\mathbb{P}(A_{k,N,j})\int_{[x_j^N,b]} \frac{\mathbb{P}(x\in F_k\;|\;x_j^N\in F_k)}{\mathbb{P}(x\in F_k)}\dd\nu(x).\nonumber\\
        &=\sum_j\mathbb{P}(A_{k,N,j})\int_{[x_j^N,b]} p_k(x_j^N,x)\dd\nu(x)\geq P(A_{k,N}) M.
    \end{align*}
    By taking $N$ to infinity, we see that $\mathbb{P}(A_{k})\leq \tfrac{1}{M}$, 
    which establishes \eqref{eq:ymk+-covered'}. As mentioned before, \eqref{eq:ymk+-covered'} is established for $Y_{M,k}^{-}$ in the same manner except we number the dyadic intervals from right to left and use part $ii)$ of \cref{lemma:cruxlemma}.
 
\end{proof}

\begin{proof}[Proof of Theorem \ref{thm:analyticsetcoverediffcapacityzero}]
    As discussed after the statement of \cref{thm:analyticsetcoverediffcapacityzero}, if $J(\nu)=\infty$ for all $\nu\in\PPC(A)$, then $A\setminus X_0$ is covered by $E_{\rrr}(\omega)$ almost surely. If  $C\subset A\cap X_0 $ is compact, \cref{lemma:ifcapacityzerothenmeasureofgoodtendsto1} and \cref{prop:no-billard-cover'} imply that for every $\nu \in \PP(C)$, there exists a Borel set $C_{\nu}\subset C$ with $\nu(C_{\nu})=1$ which satisfies $\mathbb{P}(C_{\nu}\subset E_{\rrr})=1$. \cref{prop:ae-cover-suffices'} then yields that for every compact $C\subset A\cap X_0$ we have $\mathbb{P}(C\subset E_{\rrr})=1$, and finally \cref{prop:analyticcoveredifeverycompactcovered} yields that $\mathbb{P}(A\cap X_0\subset E_{\rrr})=1$. Therefore the union $A=(A\setminus X_0)\cup (A\cap X_0)$ is covered by $E_{\rrr}(\omega)$ almost surely. 
\end{proof}

\section{Proof of Theorem \ref{thm:main}}\label{sec:proofofthmmain}
The purpose of this section is to deduce \cref{thm:main} from \cref{thm:unbounded-moments}. The proof is based on the following observation, which is certainly standard and well known in the literature, but we include the proof for completeness.
\begin{lemma}\label{lemma:secondmomentcomparabletoexp}
    Let $\mu\in\PP(\R^d)$ and suppose $A\subset \R^d$ is such that $\sum_{k}\mu(B(x,r_k))^2\leq M$ for every $x\in A$. Then for every $x,y\in A$ and every $k\in \N$, 
    \begin{equation*}
    p_k(x,y)\approx \exp \left( \sum_{n=1}^k\mu(B(x,r_n)\cap B(y,r_n)) \right).    
    \end{equation*}
\end{lemma}
\begin{proof}
  Recall \eqref{eqn:measuresboundedawayfrom1} and write $L=\inf_{z\in A}(1-\mu(B(z,r_1)))>0$.  Note that
    \begin{align*}
        p_k(x,y)&=\prod_{n=1}^k\frac{1-\mu(B(x,r_n)-\mu(B(y,r_n)+\mu(B(x,r_n)\cap B(y,r_n))}{(1-\mu(B(x,r_n)))(1-\mu(B(y,r_n)))}\\
        &=\prod_{n=1}^k\left(1+\frac{\mu(B(x,r_n)\cap B(y,r_n))-\mu(B(x,r_n))\mu(B(y,r_n))}{(1-\mu(B(x,r_n)))(1-\mu(B(y,r_n)))}\right).
    \end{align*}
    Applying the fact that $1+x=\exp(x+O(x^2))$ and noting that
    \begin{align*}
        &\sum_{n=1}^k\left(\frac{\mu(B(x,r_n)\cap B(y,r_n))-\mu(B(x,r_n))\mu(B(y,r_n))}{(1-\mu(B(x,r_n)))(1-\mu(B(y,r_n)))}\right)^2\\
        &\leq \frac{\sum_{n=1}^\infty\mu(B(x,r_n))^2}{L^4}\leq \frac{M}{L^4}<\infty,
    \end{align*}
    we have
    \begin{align}\label{eq:approxforp_k(x,y)}
        p_k(x,y)&\approx\exp\left(\sum_{n=1}^k \frac{\mu(B(x,r_n)\cap B(y,r_n))-\mu(B(x,r_n))\mu(B(y,r_n))}{(1-\mu(B(x,r_n)))(1-\mu(B(y,r_n)))}\right) \nonumber\\
        &\approx \exp\left(\sum_{n=1}^k \frac{\mu(B(x,r_n)\cap B(y,r_n))}{(1-\mu(B(x,r_n)))(1-\mu(B(y,r_n)))}\right),
    \end{align}
    where on the second approximation we used the fact that sum of the negative terms on the first line are bounded in absolute value from above by $2M/L^2$. 
    We then calculate 
    \begin{align*}
       0 &\leq \frac{\mu(B(x,r_n)\cap B(y,r_n))}{(1-\mu(B(x,r_n)))(1-\mu(B(y,r_n)))}\\
        &\leq  \mu(B(x,r_n)\cap B(y,r_n))\left( 1+ \frac{\mu(B(x,r_n))+\mu(B(y,r_n))}{(1-\mu(B(x,r_n)))(1-\mu(B(y,r_n)))} \right)\\
        &\leq \mu(B(x,r_n)\cap B(y,r_n))\left( 1+ \frac{\mu(B(x,r_n))+\mu(B(y,r_n))}{L^2} \right)\\
        &\leq \mu(B(x,r_n)\cap B(y,r_n))+ \frac{\mu(B(x,r_n))^2+ \mu(B(y,r_n))^2 }{L^2},
    \end{align*}
    which implies that the sum in \eqref{eq:approxforp_k(x,y)} is bounded above by 
    \[
    \sum_{n=1}^k \mu(B(x,r_n)\cap B(y,r_n))+\frac{2M}{L^2},
    \]
    which implies the claim.
\end{proof}

Combining the previous lemma with \cref{thm:unbounded-moments} and a uniformisation argument will allow us to show that if $\sum_{n=1}^{\infty}\mu(B(x,r_n))^2<\infty$, for all $x\in A$, then $\capa_{\mu,\rrr}(A)=0$ if and only if $J(\nu)=\infty$ for all $\nu\in \PPC(A)$. To upgrade this to \cref{thm:main} we show that for any $A\subset \R$, the set
\begin{equation*}
    \left\{x\in A\colon \sum_{n=1}^{\infty}\mu(B(x,r_n))^2=\infty\right\},
\end{equation*}
is automatically covered almost surely. This is the purpose of the following slightly stronger proposition. In the following, we will denote by $X_{atom}=\{x\in X\colon \mu(\{x\})>0\}$.
\begin{proposition}\label{prop:1r-covering}
    Let $\mu\in\PP(\R)$, $\varepsilon>0$ and let $C\subset X\setminus (X_{atom}\cup X_{atomish})$ be a compact set with the property that $\sum_k\mu(B(x,r_k))^{1+\varepsilon}=\infty$ for every $x\in C$. Then 
    \begin{equation*}
        C\subset E_{\rrr}
    \end{equation*}
    almost surely.
\end{proposition}
\begin{proof}
If $C=\emptyset$, we have nothing to prove. Otherwise, since $C$ is compact and $\mu$ does not have atoms at $C$, we have that $\sup_{x\in C}\mu(B(x,r))\to 0$ as $r\to 0$. Thus we may assume that $\sup_{x\in C}\mu(B(x,r_1))<\delta$, where $\delta>0$ is small enough so that 
\begin{equation}\label{eq:sometriviallogineq'}
    \log(1-x)\leq \frac{-x}{2}
\end{equation}
for every $0<x<\delta$. 

Define the sets
\begin{align*}
    &C^+\coloneqq \{x\in C\colon \sum_k \mu([x,x+r_k))^{1+\varepsilon}=\infty \},\\
    &C^-\coloneqq \{x\in C\colon \sum_k \mu((x-r_k,x])^{1+\varepsilon}=\infty \},
\end{align*}
and observe that $C=C^+\cup C^-$.
We will show that $C^+\subset E_{\rrr}$ almost surely and the argument for $C^-$ is symmetric.
Let $0<\eta<\varepsilon$ and for each $k\in \N$, let 
\[
    \Gamma_k^+=\{x\in C\colon \mu([x,x+r_k))^{1+\varepsilon} \geq k^{-(1+\eta)} \}.
\]  
By definition of $C^+$, we have that $C^+\subset \limsup_{k\to \infty}\Gamma_k^+$. 
    Note that the condition $x\in \Gamma_k^+$ implies (by \eqref{eq:sometriviallogineq'}) that 
    \begin{align*}
    \mathbb{P}(\omega_j\not \in [x,x+r_k) \,\forall\, j\leq k)&=(1-\mu([x,x+r_k)))^k
    =\exp\left( k \log (1-\mu([x,x+r_k)) \right)\\
    &\leq \exp \left( -\frac{1}{2}k \mu([x,x+r_k)) \right)
    \leq \exp \left( -\frac{1}{2}k^{s} \right),
\end{align*}
where $s\coloneqq 1-\frac{1+\eta}{1+\varepsilon}>0$.

For each $k\in \N$, let $\{x_{k,l}\}_{l=1}^{N_k} \subset \Gamma_k^+$ be a maximal $2r_k$-separated subset. Since the balls $B(x_{k,l},r_k)_l$ are disjoint,
\begin{align*}
    1\geq \sum_{l=1}^{N_k}\mu(B(x_{k,l},r_k))\geq N_k\, k^{-\frac{1+\eta}{1+\varepsilon}},
\end{align*}
which yields $N_k\leq k^{\frac{1+\eta}{1+\varepsilon}}$.
Moreover if we let $\{\Tilde{x}_{k,l}\}_{l=1}^{\Tilde{N}_k} \subset \Gamma_k^+$ be a maximal $r_k/3$-separated subset, then $\Tilde{N}_k\leq 13N_k$. This follows trivially, since $\bigcup_{l=1}^{N_k}\overline{B}(x_{k,l},2r_k)$ is a cover for $\Gamma_k^+$, and a maximal $r_k/3$-packing of  $\overline{B}(x_{k,l},2r_k)$ has at most $13$ points.

Since $C\subset \R\setminus (X_{atom}\cup X_{atomish})$, the map $x\mapsto \mu([x,x+r_k))$ is continuous on the compact set $C$, hence $\Gamma_k^+\subset C$ is closed and thus also compact. Thus for each $l=1,\ldots, \Tilde{N}_k$, 
we have that 
\begin{equation*}
    y_{k,l}\coloneqq \inf (\overline{B}(\Tilde{x}_{k,l},r_k/3)\cap \Gamma_k^+)\in \Gamma_k^+,
\end{equation*}
and $\overline{B}(\Tilde{x}_{k,l},r_k/3)\cap \Gamma_k^+\subset I_{k,l},$ where $I_{k,l}=[y_{k,l},y_{k,l}+r_k)$ Thus 
\[
\Gamma_k^+\subset \bigcup_{l=1}^{\Tilde{N}_k}I_{k,l}.
\]
Observe that for $j\leq k$, if $\omega_j\in I_{k,l}$, then $I_{k,l}\subset B(\omega_j,r_k)\subset B(\omega_j,r_j)$.
Thus the event $\{\omega\colon \Gamma_k^+\cap F_k(\omega)\neq \emptyset \}$ is contained in the event $\{\omega\colon\exists l\in\{1,\ldots ,\Tilde{N}_k\},\, I_{k,l}\cap \{\omega_j \}_{j=1}^k=\emptyset\}$.
Therefore
\begin{align*}
    \mathbb{P}(\Gamma_k^+\cap F_k\neq \emptyset )&\leq
    \mathbb{P}(\exists l\in\{1,\ldots ,\Tilde{N}_k\},\, I_{k,l}\cap \{\omega_j \}_{j=1}^k=\emptyset)\\
    &\leq\sum_{l=1}^{\Tilde{N}_k}\mathbb{P}(\omega_j\not \in [y_{k,l},y_{k,l}+r_k[\, \forall\, j)\\
    &\leq \Tilde{N}_k\exp \left( -\frac{1}{2}k^{s} \right)\leq 13k^{\frac{1+\eta}{1+\varepsilon}}\exp \left( -\frac{1}{2}k^{s} \right).
\end{align*}
 Since this is summable in $k$, Borel--Cantelli lemma implies that for $\mathbb{P}$-almost every $\omega$, there is $k(\omega)\in \N$ such that $\Gamma_k^+\cap F_k(\omega)=\emptyset$ for every $k\geq k(\omega)$. In other words, for every $k\geq k(\omega)$,
\[
\Gamma_k^+\subset \bigcup_{j=1}^k B(\omega_j,r_j)
\]
which means that 
\[
C^+\subset\limsup_{k\to \infty}\Gamma_k^+\subset \bigcup_{j=1}^\infty B(\omega_j,r_j)
\]
almost surely. By repeating the process starting from indices larger than one, we obtain 
the claim.
\end{proof}
We are now ready to prove \cref{thm:main}.
For an analytic set $A\subset \R$, we denote by
\begin{equation}\label{eq:A'-def}
  A'\coloneqq A\cap X_{\mu,\rrr}=\{x\in A\colon \sum_{k=1}^{\infty}\mu(B(x,r_k))^2<\infty\}. 
\end{equation}
\begin{proof}[Proof of \cref{thm:main}]
    By \cref{prop:1r-covering}, 
    every compact subset of $A\setminus (A'\cup X_{atom}\cup X_{atomish})$ is covered almost surely. \cref{prop:analyticcoveredifeverycompactcovered} then yields that the whole Borel set 
   \[
   A\setminus (A'\cup X_{atom}\cup X_{atomish})=(A\setminus A')\setminus (X_{atom}\cup X_{atomish})
   \]
    is covered almost surely. By Borel--Cantelli lemma, the countable set 
    \[
    (A\setminus A')\cap (X_{atom}\cup X_{atomish})
    \]
    is covered almost surely (recalling the definition of $A'$). Thus $A\setminus A'$ is automatically covered almost surely, which implies that $A$ is covered almost surely if and only if $A'$ is covered almost surely.
    By \cref{thm:analyticsetcoverediffcapacityzero}, it then suffices to show that $\capa_{\mu,\rrr}(A')=0$ if and only if $J(\nu)=\infty$ for all $\nu\in\PPC(A')$.

    Assume first that there exists $\nu\in\PPC(A')$, such that
    \begin{equation*}
        J(\nu)=\lim_{k\to\infty}\iint p_k(x,y)\dd\nu(y)\dd\nu(x)<\infty.
    \end{equation*}
    By \cref{eq:A'-def}, there is $M<\infty$ and a set $A''$ which satisfies $\nu(A'')>0$ and
    \begin{equation*}
        \sum_{k=1}^{\infty}\mu(B(x,r_k))^2\leq M,
    \end{equation*}
    for all $x\in A''$. Let $\nu'=\frac{\nu|_{A''}}{\nu(A'')}$
    and note that $J(\nu')\leq \nu(A'')^{-2}J(\nu)$.
    By \cref{lemma:secondmomentcomparabletoexp}, the continuity of the exponential function and the monotone convergence theorem, we have
    \begin{equation*}
        \iint\exp \left( \sum_{n=1}^\infty\mu(B(x,r_k)\cap B(y,r_k)) \right)\dd\nu'(y)\dd\nu'(x)\approx\lim_{k\to\infty}\iint p_k(x,y)\dd\nu'(y)\dd\nu'(x)<\infty.
    \end{equation*}
    This shows that $\capa_{\mu,\underline{r}}(A')>0$. The other direction is essentially identical starting from the assumption that $\capa_{\mu,\underline{r}}(A')>0$.
\end{proof}
We note that \cref{prop:1r-covering-prop} follows from \cref{prop:1r-covering} together with \cref{prop:analyticcoveredifeverycompactcovered} in a similar way to the previous proof using the fact that $X_{atom}$ and $X_{atomish}$ are countable.

\section{Polynomial radii in \texorpdfstring{$\R^d$}{Rd}}\label{sec:criticalexponent}
Unfortunately, we were unable to extend \cref{prop:1r-covering} to higher dimensions, so in order to prove \cref{thm:seuret-generalization} in $\R^d$ we need the following weaker version which works in $\R^d$. The proof of the proposition is very similar to that of \cref{prop:1r-covering}.

\begin{proposition}\label[proposition]{prop:1+epsilonk-covering}
    Let $\mu\in\PP(\R^d)$, $0<c<1$ and let $C\subset \R^d\setminus X_{atom}$ be a compact set with the property that $\sum_k\mu(B(x,cr_k))^{1+\varepsilon}=\infty$ for every $x\in C$.Then 
    \begin{equation*}
        C\subset \limsup_{k\to\infty}B(\omega_k,r_k),
    \end{equation*}
    almost surely.
\end{proposition}

\begin{proof}
    If $C=\emptyset$, there is nothing to prove. Otherwise, as in the proof of \cref{prop:1r-covering}, we may assume that $\sup_{x\in C}\mu(B(x,cr_1))<\delta$, where $\delta>0$ is small enough so that 
    \begin{equation}\label{eq:sometriviallogineq}
        \log(1-x)\leq \frac{-x}{2}
    \end{equation}
    for every $0<x<\delta$. 

    Let $0<\eta<\varepsilon$ and for each $k\in \N$, let 
    \[
    \Gamma_k=\{x\in C\colon \mu(B(x,cr_k))^{1+\varepsilon} \geq k^{-(1+\eta)} \}.
    \]
    By our assumption, $C=\limsup_{k\to \infty}\Gamma_k$. Note that since $\rrr$ is decreasing, the condition $x\in \Gamma_k$ implies that 
    \[
    \sum_{j=1}^k\mu(B(x,cr_j))\geq k^{1-\frac{1+\eta}{1+\varepsilon}}\eqqcolon k^{s},
    \]
    Therefore, if we denote by $\Tilde{F}_k(\omega)=\R^d\setminus \bigcup_{n=1}^kB(\omega_n,cr_n)$, then by \eqref{eq:sometriviallogineq}, 
    \begin{align*}
        \mathbb{P}(x\in \Tilde{F}_k)&=\prod_{j=1}^k(1-\mu(B(x,cr_j)))
        =\exp\left( \sum_{j=1}^k \log (1-\mu(B(x,cr_j))) \right)\\
        &\leq \exp \left( -\frac{1}{2}\sum_{j=1}^k \mu(B(x,cr_j)) \right)
        \leq \exp \left( -\frac{1}{2}k^{s} \right)
    \end{align*}
    for $x\in \Gamma_k$.
    For each $k\in \N$, let $\{x_{k,l}\}_{l=1}^{N_k} \subset \Gamma_k$ be a maximal $cr_k$-separated subset. As before, since the balls $B(x_{k,l},cr_k)_l$ are disjoint,
    \begin{align*}
        1\geq \sum_{l=1}^{N_k}\mu(B(x_{k,l},cr_k))\geq N_k\, k^{-\frac{1+\eta}{1+\varepsilon}},
    \end{align*}
    which yields $N_k\leq k^{\frac{1+\eta}{1+\varepsilon}}$. Moreover if we let $\{\Tilde{x}_{k,l}\}_{l=1}^{\Tilde{N}_k} \subset \Gamma_k$ be a maximal $(1-c)r_k$-separated subset, then, since $\bigcup_{l=1}^{N_k}\overline{B}(x_{k,l},cr_k)$ is a cover for $\Gamma_k$, and a maximal $(1-c)r_k$-packing of  $\overline{B}(x_{k,l},2cr_k)$ has at most $\left(\frac{4c}{(1-c)}\right)^d$ points, we have that $\Tilde{N}_k\lesssim N_k$.

    Observe now that if $z\in F_k(\omega)\cap \Gamma_k$, then for some $l=1,\ldots, \Tilde{N}_k$, $|z-\Tilde{x}_{k,l}|\leq (1-c)r_k$ and for all $j=1,\ldots,k$, 
    $|z-\omega_j|\geq r_j$. This implies that for every $j=1,\ldots ,k$
    \begin{align*}
        r_j\leq|\omega_j-z|\leq |z-\Tilde{x}_{k,l}|+|\Tilde{x}_{k,l}-\omega_j|\leq (1-c)r_k+|\Tilde{x}_{k,l}-\omega_j|,
    \end{align*}
    which in turn implies that $|\Tilde{x}_{k,l}-\omega_j|\geq r_j-(1-c)r_k\geq cr_j$, hence $\Tilde{x}_{k,l}\in \Tilde{F}_k(\omega)$. In particular, the event $\{\omega\colon \Gamma_k\cap F_k(\omega)\neq \emptyset \}$ is contained in the event $\bigcup_{l=1}^{\Tilde{N}_k}\{\omega\colon\Tilde{x}_{k,l}\in \Tilde{F}_k(\omega) \}$. From this we obtain that 
    \begin{equation}\label{eq:no-cover-prob}
        \mathbb{P}(\Gamma_k\cap F_k\neq \emptyset )\leq \sum_{l=1}^{\Tilde{N}_k}\mathbb{P}(\Tilde{x}_{k,l}\in \Tilde{F}_k)\lesssim k^{\frac{1+\eta}{1+\varepsilon}}\exp\left(-\frac{1}{2}k^{s}\right),
    \end{equation}
    for all large enough $k$. Since this is summable in $k$, Borel--Cantelli lemma implies that for $\mathbb{P}$-almost every $\omega$, there is $k(\omega)\in \N$ such that $\Gamma_k\cap F_k=\emptyset$ for every $k\geq k(\omega)$. The claim now follows exactly as in the end of the proof of \cref{prop:1r-covering}.
\end{proof}
\begin{remark}
By making a slightly more careful calculation, the constant $0<c<1$ in the proposition can be made scale dependent. Indeed, by making the implicit constant in \cref{eq:no-cover-prob} explicit, for a fixed $\varepsilon>0$, one can replace the condition in the proposition by the requirement that
\begin{equation*}
    \sum_{n=1}^{\infty}\mu(B(x,(1-\delta_n)r_n))^{1+\varepsilon}=\infty,
\end{equation*}
for all $x\in C$, where $\delta_n\gtrsim \exp(-n^{t})$, for $t<s$. This shows that, a counterexample to \cref{prop:1r-covering-prop} in $\R^d$ would have to be rather extreme. For example, one can show that this stronger version of the proposition implies that if $\mu\in\PP(\R^d)$ satisfies $\ldimloc(\mu,x)>0$ for all $x\in C$, then \cref{prop:1r-covering} holds for $\mu$. We leave the details to the interested reader.
\end{remark}
We are now ready to prove \cref{thm:seuret-generalization}, which will follow immediately from the following somewhat stronger proposition, which gives additional information on which parts of $X$ are covered. For a fixed measure $\mu$ and a given $\alpha\in\R$, we denote by
\begin{equation*}
    \Delta^{\mathrm{loc}}_{\mu}(\alpha)=\{x\in \spt{\mu}\colon \ldimloc(\mu,x)=\alpha\},
\end{equation*}
and 
\begin{equation*}
    \underline{\Delta}^{\mathrm{loc}}_{\mu}(\alpha)=\{x\in \spt{\mu}\colon \ldimloc(\mu,x)\leq\alpha\}.
\end{equation*}

\begin{proposition}\label{prop:seuret-generalization}
    Let $\mu\in\PP(\R^d)$, $c,t>0$ and let $\underline{r}=(cn^{-t})_{n=1}^{\infty}$.
    \begin{enumerate}
        \item\label{it:seuret1} If $\frac{1}{t}<\alpha$ and $\Delta^{\mathrm{loc}}_{\mu}(\alpha)\neq \emptyset$, then $\Delta^{\mathrm{loc}}_{\mu}(\alpha)$ is not covered by $E_{\rrr}$ almost surely.
        \item\label{it:seuret2} If $\frac{1}{t}>\alpha$, then $\underline{\Delta}^{\mathrm{loc}}_{\mu}(\alpha)$ is covered by $E_{\rrr}$ almost surely.
    \end{enumerate}
\end{proposition}
\begin{proof}
    For \cref{it:seuret1}, let $\frac{1}{t}<s<\alpha$, and note that for any $x\in\Delta^{\mathrm{loc}}_{\mu}(\alpha)$, we have
    \begin{equation*}
        \mu(B(x,cn^{-t}))\leq n^{-st},
    \end{equation*}
    for all large enough $n$. Therefore $\sum_{n=1}^{\infty}\mu(B(x,cn^{-t}))<\infty$, which gives the claim.

    For \cref{it:seuret2}, let $\alpha<s<\frac{1}{t}$, $0<c_1<1$, and let $\varepsilon>0$ be small enough so that 
    \[
    \gamma\coloneqq st(1+\varepsilon)<1.
    \]
    Note that for each $x\in \underline{\Delta}^{\mathrm{loc}}_{\mu}(\alpha)$, the inequality
    \begin{equation*}
        (\mu(B(x,c_1cn^{-t})))^{1+\varepsilon}\geq (n^{-st})^{1+\varepsilon}=n^{-\gamma}
    \end{equation*}
    holds for infinitely many $n\in \N$. Since $\gamma<1$, it is standard (see e.g.  proof of \cite[Lemma~5.1]{JarvenpaasMyllyojaStenflo2025}) that this implies
    \begin{equation*}
        \sum_{n=1}^{\infty}\mu(B(x,c_1cn^{-t}))^{1+\varepsilon}=\infty,
    \end{equation*}
    and the claim follows from \cref{prop:1+epsilonk-covering}
\end{proof}
We emphasise that the proof of \cref{thm:seuret-generalization} is completely self-contained and does not use the machinery developed for the proof of \cref{thm:main}. However, the critical case of $t=\frac{1}{\alpha}$ is once again very subtle, and certainly requires the full power of \cref{thm:main}. In the following section, we handle this critical exponent for natural measures on self-conformal sets---in particular for the Hausdorff measure on the Cantor set---by showing that the critical constant for the covering problem depends in a subtle way on the average densities of the measure. For measures which are not Ahlfors regular, determining the critical constant is likely to be more difficult.

\section{Dvoretzky covering problem for self-conformal sets}\label{sec:self-conformal}
In this section, we study the Dvoretzky covering problem on self-conformal sets and prove \cref{thm:self-conformal}. Let us recall the setup. Fix a finite index set $\Lambda$. An \emph{iterated function system (IFS)} $\{g_i\}_{i\in\Lambda}$ is a finite collection of contracting self-maps on $\R$. Any IFS admits a unique non-empty and compact \emph{attractor} $X$, which is the set that satisfies
\begin{equation*}
    X=\bigcup_{i\in\Lambda}g_i(X).
\end{equation*}
Going forward, we assume $\#\Lambda>1$ to exclude the trivial case when $X$ is a single point. We will also assume that the IFS is \emph{self-conformal}, that is for each $i\in\Lambda$, the function $g_i\colon [0,1]\to [0,1]$ is a contracting $C^{1+\varepsilon}$ function with non-vanishing derivative. In this case, we call the attractor $X$ the \emph{self-conformal set}. We further assume that the IFS satisfies the \emph{strong separation condition (SSC)}, that is $g_i(X)\cap g_j(X)=\emptyset$ for $i\ne j$.

Every IFS admits a natural coding via the \emph{symbolic space} $\Sigma\coloneqq \Lambda^{\N}$. We denote by $\Sigma_n=\Lambda^n$ the \emph{words of length $n$} and by $\Sigma_*=\bigcup_{n=0}^{\infty}\Sigma_n$ the collection of all \emph{finite words}, where $\Sigma_0$ is the singleton consisting of the \emph{empty symbol} $\varnothing$. For any $\mtt{i}=i_1i_2\ldots i_n\in\Sigma_*$, we denote by
\begin{equation*}
    g_{\mtt{i}}=g_{i_1}\circ g_{i_2}\circ\ldots\circ g_{i_n}.
\end{equation*}
We use the convention that $g_{\varnothing}=\mathrm{id}$ is the identity map. For words $\mtt{i},\mtt{j}\in\Sigma_*$ we denote by $\mtt{i}\mtt{j}=i_1\cdots i_{|\mtt{i}|}j_1\cdots j_{|\mtt{j}|}$ the \emph{concatenation} of $\mtt{i}$ and $\mtt{j}$. Here and hereafter $|\mtt{i}|$ denotes the \emph{length} of $\mtt{i}$, that is $|\mtt{i}|$ is the unique natural number which satisfies $\mtt{i}\in\Sigma_{|\mtt{i}|}$. For $\mtt{i}\in\Sigma$, and $n\in\N$ we let $\mtt{i}|_n=i_1i_2\cdots i_n$ denote the \emph{restriction} of $\mtt{i}$ onto the first $n$ symbols. The natural topology on $\Sigma$ is generated by the \emph{cylinder sets}
\begin{equation*}
    [\mtt{i}]=\{\mtt{j}\in\Sigma\colon \mtt{j}|_{|\mtt{i}|}=\mtt{i}\},
\end{equation*}
with $\mtt{i}\in\Sigma_*$. The \emph{natural projection} $\pi\colon \Sigma\to X$ is then given by
\begin{equation*}
    \pi(\mtt{i})=\lim_{n\to\infty}g_{\mtt{i}|_n}(0).
\end{equation*}

For a continuous function $f\colon \R\to\R$, we denote by $\|f\|=\max_{x\in X}|f(x)|$. The following lemma is standard, see e.g. \cite{MauldinUrbanski1996}.
\begin{lemma}\label{lemma:conformal}
    There are constants $C\geq 1$ and $\delta>0$, such that the following hold.
    \begin{enumerate}[label={(\roman*)}]
        \item For all $\mtt{i}\in\Sigma_*$ and $x,y\in X$, we have $|g_{\mtt{i}}'(x)|\leq C|g_{\mtt{i}}'(y)|$.
        \item For any $x,y,z\in X$ with $|x-y|\leq \delta$, we have
        \begin{equation*}
            C^{-1}|g_{\mtt{i}}'(z)|\leq \frac{|g_{\mtt{i}}(x)-g_{\mtt{i}}(y)|}{|x-y|}\leq C|g_{\mtt{i}}'(z)|,
        \end{equation*}
        for all $\mtt{i}\in \Sigma_*$.
        \item For all $\mtt{i}\in\Sigma_*$, we have $C^{-1}\|g_{\mtt{i}}'\|\leq \diam(g_{\mtt{i}}(X))\leq C\|g_{\mtt{i}}'\|$.
    \end{enumerate}
\end{lemma}
The first item in the previous lemma is commonly called the \emph{bounded distortion principle}.
We remark that under a suitable choice of metric on $\Sigma$, namely by letting
\begin{equation*}
    d(\mtt{i},\mtt{j})=\|g_{\mtt{i}\wedge\mtt{j}}'\|,
\end{equation*}
where $\mtt{i}\wedge\mtt{j}=\max\{k\in\N\colon \mtt{i}|_k=\mtt{j}|_k\}$, the function $\pi$ is Lipschitz, and if the IFS satisfies the SSC, then $\pi$ is bi-Lipschitz. Moreover, the topology induced by this metric coincides with the one generated by the cylinder sets.
Notice as well that by the bounded distortion principle, compactness of $X$ and continuity of the functions $g_i'$, there are numbers $0<\lambda_{\min}\leq\lambda_{\max}<1$, such that for all $n\in\N$ and $\mtt{i}\in\Sigma_n$,
\begin{equation*}
     \lambda_{\min}^n\leq \|g_{\mtt{i}}'\|\leq \lambda_{\max}^n.
\end{equation*}
Using the SSC together with \cref{lemma:conformal}, we may fix a constant $\delta>0$, such that for all $n\in\N$ and $\mtt{i},\mtt{j}\in\Sigma_n$,
\begin{equation}\label{eq:ssc}
    \dist(g_{\mtt{i}}(X),g_{\mtt{j}}(X))\geq \delta\diam(g_{\mtt{i}}(X)).
\end{equation}

\subsection{Natural measures on self-conformal sets}
The centres of our balls will be distributed according to the natural measure on $X$, which is the normalised $s$-Hausdorff measure on $X$. This measure is equal to the so called $s$-conformal measure on $X$ whose construction we next recall. Let $P$ denote the \emph{pressure function} defined by
\begin{equation*}
    P(t)=\lim_{n\to\infty}\frac{1}{n}\sum_{\mtt{i}\in\Sigma_n}\|g_i'\|^t,
\end{equation*}
where $\|g_i'\|=\max_{x\in X}|g_i'(x)|$. The limit exists by Fekete's lemma and the function $P$ is strictly decreasing and continuous, and has a unique zero which we denote by $s$. This value is the Hausdorff dimension of $X$, that is $\dimh X=s$. Moreover, there exists a unique Borel probability measure $\mu$ which satisfies
\begin{equation*}
    \mu(B)=\sum_{i\in\Lambda}\int_{g_i^{-1}(B)}|g_i'(x)|^s\dd\mu,
\end{equation*}
for all Borel sets $B$. We call this measure the \emph{$s$-conformal measure}. Alternatively, one may view $\mu$ as the pushforward of the unique Gibbs measure $\nu$ on $\Sigma$ associated to the potential function $\mtt{i}\mapsto\log|g_{i_1}'(\pi(\sigma \mtt{i}))|^s$, where $\sigma\colon \Sigma\to\Sigma$ denotes the \emph{left shift} defined by
\begin{equation*}
    \sigma\mtt{i}=i_2i_3\ldots.
\end{equation*}
The measure $\nu$ satisfies the Gibbs condition, i.e. there is a constant $C\geq 1$, such that for all $n\in\N$ and $\mtt{i}\in\Sigma$,
\begin{equation*}
    C^{-1} \prod_{j=1}^k|g_{i_j}'(\pi(\sigma^j\mtt{i}))|^s\leq \nu([\mtt{i}])\leq C \prod_{j=1}^k|g_{i_j}'(\pi(\sigma^j\mtt{i}))|^s.
\end{equation*}
By the bounded distortion principle and the chain rule, this implies that there is a constant $C\geq 1$, such that for all $\mtt{i}\in\Sigma_*$,
\begin{equation*}
    C^{-1} \|g_{\mtt{i}}'\|^s\leq \mu(g_{\mtt{i}}(X))\leq C\|g_{\mtt{i}}'\|^s.
\end{equation*}
We remark that $\mu$ is equivalent with the restriction of the $s$-dimensional Hausdorff measure on $X$. For the remainder of the section, $\mu$ always refers to the fixed $s$-conformal measure.

\subsection{Pointwise analogue of the capacity for \texorpdfstring{$s$}{s}-conformal measures}
The aim of this section is to use the self-conformal structure of $s$-conformal measures to rephrase the capacity condition in our main result in terms of a pointwise condition which is easier to use in this setting. We need the following technical lemma, which shows that for (non-atomic) measures supported on self-conformal sets, the lower local dimension can be calculated along what we call \emph{symbolic annuli} instead of balls or cylinder sets.

For $x\in X$ we adopt the notation $\mtt{i}_x=\pi^{-1}(x)$. For $x\in X$ and $k\in\N$, we denote by $A_k(x)\coloneqq g_{\mtt{i}_x|_k}(X)\setminus g_{\mtt{i}_x|_{k+1}}(X)$ the \emph{scale $k$ symbolic annulus} at $x$.
\begin{lemma}\label{l:annulus-dim}
    Let $\nu\in\mathcal{P}(X)$ and assume that $\nu$ does not have an atom at $x\in X$. Then
    \begin{equation*}
        \ldimloc(\nu,x)=\liminf_{k\to\infty}\frac{\log\nu(A_k(x))}{\log \|g_{\mtt{i}_x|_k}'\|}.
    \end{equation*}
\end{lemma}
\begin{proof}
    Let $\nu\in\mathcal{P}(X)$ and $x\in X$ and note that since $X$ satisfies the SSC, we have
    \begin{equation*}
        \ldimloc(\nu,x)=\liminf_{n\to\infty}\frac{\log\nu(g_{\mtt{i}_x|_n}(X))}{\log\|g_{\mtt{i}_x|_n}'\|}.
    \end{equation*}
    Since $A_k(x)\subset g_{\mtt{i}_x|_k}(X)$ for all $k\in\N$, the upper bound in the claim always holds. Assume now that
    \begin{equation*}
        \ldimloc(\nu,x)<s<t<\liminf_{k\to\infty}\frac{\log\nu(A_k(x))}{\log \|g_{\mtt{i}_x|_k}'\|}.
    \end{equation*}
    Fix $M>\sum_{\ell=1}^\infty\lambda_{\max}^{\ell t}$ and choose some large $k\in\N$, such that
    \begin{equation*}
        \nu(g_{\mtt{i}|_{k}}(X))\geq \|g_{\mtt{i}_x|_k}'\|^s\geq M\|g_{\mtt{i}_x|_k}'\|^t,
    \end{equation*}
    and for all $\ell\in\N$,
    \begin{equation*}
        \nu(A_{k+\ell}(x))\leq \|g_{\mtt{i}_x|_{k+\ell}}'\|^t.
    \end{equation*}
    Note that $\{x\}=g_{\mtt{i}|_k}(X)\setminus \bigcup_{\ell=1}^{\infty}A_{k+\ell}(x)$, and therefore
    \begin{align*}
        \nu(\{x\})&=\nu(g_{\mtt{i}|_k}(X))-\sum_{\ell=1}^\infty\nu(A_{k+\ell}(x))\\
        &\geq M\|g_{\mtt{i}_x|_k}'\|^t-\|g_{\mtt{i}_x|_k}'\|^t\sum_{\ell=1}^{\infty}\|g_{\mtt{i}_x|_{\ell}}'\|^t\geq \|g_{\mtt{i}_x|_k}'\|^t\left(M-\sum_{\ell=1}^\infty\lambda_{\max}^{\ell t}\right)>0,
    \end{align*}
    that is $\nu$ has an atom at $x$, which is a contradiction.
\end{proof}

Let us next describe the setting where the pointwise reduction of the Billard's condition works. For a sequence $\rrr$, a word $\mtt{i}\in\Sigma$ and $k\in\N$, we let
\begin{equation*}
    N(\rrr,\mtt{i},k)=\{n\in\N\colon \delta \diam(g_{\mtt{i}|_k}(X))\leq r_n < \diam(g_{\mtt{i}|_{k-1}}(X))\},
\end{equation*}
$\underline{n}(\rrr,\mtt{i},k)=\min N(\rrr,\mtt{i},k)$, and $\overline{n}(\rrr,\mtt{i},k)=\max N(\rrr,\mtt{i},k)+1$, i.e. $\underline{n}(\rrr,\mtt{i},k)$ is the first $n\in\N$ satisfying $r_n < \diam(g_{\mtt{i}|_{k-1}}(X))$ and $\overline{n}(\rrr,\mtt{i},k)$ is the first $n\in\N$ satisfying $r_n<\delta \diam(g_{\mtt{i}|_{k}}(X))$. Here $\delta>0$ is the constant of \cref{eq:ssc}. We will assume that there are constants $0<C_1,C_2<\infty$, such that for all $\mtt{i}\in\Sigma$ and $k\in\N$,
\begin{equation}\label{eq:bouned-cylinder-sums}
    \sum_{n\in N(\rrr,\mtt{i},k)}\mu(B(\pi(\mtt{i}),r_n))\leq C_1,
\end{equation}
and 
\begin{equation}\label{eq:bouned-products}
    \underline{n}(\rrr,\mtt{i},k)\|g_{\mtt{i}|_{k-1}}'\|^s\leq C_2.
\end{equation}
It is easy to see using the $s$-Ahlfors regularity of $\mu$ and \cref{lemma:conformal} that for any $c>0$, the sequence $\rrr=(cn^{-\frac{1}{s}})$ satisfies both of these conditions.

The following result is the pointwise variant of the capacity condition we are after.
\begin{proposition}\label{prop:pw-billard}
    Let $\rrr$ be a sequence that satisfies \cref{eq:bouned-cylinder-sums} and \cref{eq:bouned-products}. Then $\capa_{\mu,\rrr}(X)>0$ if and only if there exists a non-atomic measure $\nu\in\mathcal{P}(X)$, such that
    \begin{equation}\label{eq:pw-billard}
        \sum_{k=1}^{\infty}\nu(A_k(x))\exp\sum_{n=1}^{\overline{n}(\rrr,\mtt{i},k)}\mu(B(x,r_n)))<\infty,
    \end{equation}
    for $\nu$-almost every $x\in X$.
\end{proposition}
\begin{proof}
    We start by noting that since $\mu$ is $s$-Ahlfors regular, $X=X_{\mu,\rrr}$, where $X_{\mu,\rrr}$ is the set in \cref{eq:finite-squares}, and since $X$ is compact, we have $\PPC(X)=\PP(X)$. Assume first that \eqref{eq:pw-billard} holds for some non-atomic measure $\nu\in\mathcal{P}(X)$. Fix a number $M\in\N$ and a set $E\subset X$ of positive $\nu$-measure, such that
    \begin{equation*}
        \sum_{k=1}^{\infty}\nu(A_k(x))\exp\sum_{n=1}^{\overline{n}(\rrr,\mtt{i}_x,k)}\mu(B(x,r_n)))<M,
    \end{equation*}
    for all $x\in E$ and let $\nu'=\frac{\nu|_E}{\nu(E)}$. Note that for any $x\in X$, the sets $A_k(x)$ form a partition of $X\setminus\{x\}$ and if $r\leq \delta\diam g_{\mtt{i}_x|_k}(X)$, then for any $y\in A_k(x)$, we have $B(x,r)\cap B(y,r)=\emptyset$, by \cref{eq:ssc}. Therefore
    \begin{align*}
        \int\int& \exp \left(\sum_{n=1}^{\infty}\mu(B(x,r_n)\cap B(y,r_n))\right)\dd\nu'(y)\dd\nu'(x)\\
        &= \int\sum_{k=0}^{\infty}\int_{A_k(x)} \exp \left(\sum_{n=1}^{\infty}\mu(B(x,r_n)\cap B(y,r_n))\right)\dd\nu'(y)\dd\nu'(x)\\
        &\leq \int\sum_{k=0}^{\infty}\int_{A_k(x)} \exp \sum_{n=1}^{\overline{n}(\rrr,\mtt{i}_x,k)}\mu(B(x,r_n))\dd\nu'(y)\dd\nu'(x)\\
        &\leq \frac{1}{\nu(E)} \int\sum_{k=0}^{\infty}\nu(A_k(x))\exp \sum_{n=1}^{\overline{n}(\rrr,\mtt{i}_x,k)}\mu(B(x,r_n))\dd\nu'(x)\leq \frac{M}{\nu(E)}<\infty,
    \end{align*}
    and thus $\capa_{\mu,\rrr}(X)>0$.

    Now let $\nu\in\mathcal{P}(X)$ and assume first that $\nu$ does not have atoms and does not satisfy \eqref{eq:pw-billard}, i.e. there is a set $E\subset X$ of positive $\nu$ measure, such that
    \begin{equation*}
        \sum_{k=1}^{\infty}\nu(A_k(x))\exp\sum_{n=1}^{\overline{n}(\rrr,\mtt{i},k)}\mu(B(x,r_n)))=\infty,
    \end{equation*}
    for all $x\in E$. On the other hand Note that if $y\in A_k(x)$ and $n\leq \underline{n}(\rrr,\mtt{i}_x,k)$, then $B(x,r_n)\setminus B(y,r_n)$ is an interval of length at most $|x-y|\leq C\|g_{\mtt{i}_x|_{k-1}}'\|$. In particular, by the $s$-Ahlfors regularity of $\mu$,
    \begin{equation*}
        \mu(B(x,r_n)\cap B(y,r_n))\geq \mu(B(x,r_n))-C^s\|g_{\mtt{i}_x|_{k-1}}'\|^s.
    \end{equation*}
    Then by \eqref{eq:bouned-cylinder-sums} and \eqref{eq:bouned-products},
    \begin{align*}
        \int\int& \exp \sum_{n=1}^{\infty}\mu(B(x,r_n)\cap B(y,r_n))\dd\nu(y)\dd\nu(x)\\
        &= \int\sum_{k=1}^{\infty}\int_{A_k(x)} \exp \sum_{n=1}^{\infty}\mu(B(x,r_n)\cap B(y,r_n))\dd\nu(y)\dd\nu(x)\\
        &\geq \int\sum_{k=1}^{\infty}\int_{A_k(x)} \exp\left( \sum_{n=1}^{\underline{n}(\rrr,\mtt{i}_x,k)}(\mu(B(x,r_n))-C^s\|g_{\mtt{i}_x|_{k-1}}'\|^s)\right)\dd\nu(y)\dd\nu(x)\\
        &\geq \int\sum_{k=1}^{\infty}\int_{A_k(x)} \exp\Bigg(\sum_{n=1}^{\overline{n}(\rrr,\mtt{i}_x,k)}\mu(B(x,r_n))\\
        &-\sum_{n\in N(\rrr,\mtt{i}_x,k)}\mu(B(x,r_n))-C^s\underline{n}(\rrr,\mtt{i}_x,k)\|g_{\mtt{i}_x|_{k-1}}'\|^s\Bigg)\dd\nu(y)\dd\nu(x)\\
        &\geq \exp(-C_1-C^sC_2)\int\sum_{k=1}^{\infty}\nu(A_k(x)) \exp\left(\sum_{n=1}^{\overline{n}(\rrr,\mtt{i}_x,k)}\mu(B(x,r_n))\right)\dd\nu(x)=\infty.
    \end{align*}
    On the other hand, if $\nu$ has an atom, then it follows easily from the fact that $\sum_{n=1}^{\infty}\mu(B(x,r_n))=\infty$ for all $x\in X$, that $I_{\mu,\rrr}(\nu)=\infty$. This shows that $\capa_{\mu,\rrr}(X)=0$.
 \end{proof}

\subsection{Average densities and proof of \texorpdfstring{\cref{thm:self-conformal}}{Theorem \ref*{thm:self-conformal}}}
Next we show that there is a conection between the sums
\begin{equation*}
    \sum_{n=1}^{\overline{n}(\rrr,\mtt{i}_x,k)}\mu(B(x,cn^{-\tfrac{1}{s}})),
\end{equation*}
and the average densities of $\mu$. For $t>0$, we let
\begin{equation*}
    \AA(x,t)=\frac{1}{-\log t}\int_{t}^1\mu(B(x,r))r^{-(s-1)}\dd r,
\end{equation*}
and recall that $\overline{\AA}(x)=\limsup_{t\to0}\AA(x,t)$ and $\underline{\AA}(x)=\liminf_{t\to0}\AA(x,t)$. Notice moreover that if $(t_k)_k$ is a decreasing sequence which satisfies $t_{k+1}\geq Ct_k$ for some $C>0$, then $\overline{\AA}(x)=\limsup_{k\to\infty}\AA(x,t_k)$ and $\underline{\AA}(x)=\liminf_{k\to\infty}\AA(x,t_k)$. We have the following lemma.
\begin{lemma}\label{l:avg-density-sum}
    For any $x\in X$, we have
    \begin{equation*}
        \lim_{k\to\infty}\left|\frac{1}{-s\log \|g_{\mtt{i}_x|_{k}}'\|}\sum_{n=1}^{\overline{n}(\rrr,\mtt{i}_x,k)}\mu(B(x,cn^{-\tfrac{1}{s}}))-c^s\AA(x,c\overline{n}(\rrr,\mtt{i}_x,k)^{-\frac{1}{s}})\right|
        =0
    \end{equation*}
\end{lemma}
\begin{proof}
    Note that for any $N\in\N$, we have
    \begin{align*}
        \int_{cN^{-\frac{1}{s}}}^c\mu(B(x,r))r^{-s-1}\dd r&=\sum_{n=1}^{N-1}\int_{c(n+1)^{-\frac{1}{s}}}^{cn^{-\frac{1}{s}}}\mu(B(x,r))r^{-s-1}\dd r\\
        &\leq \sum_{n=1}^{N-1}\mu(B(x,cn^{-\frac{1}{s}}))\int_{c(n+1)^{-\frac{1}{s}}}^{cn^{-\frac{1}{s}}}r^{-s-1}\dd r\\
        &\leq\frac{1}{c^ss}\sum_{n=1}^{N}\mu(B(x,cn^{-\frac{1}{s}})),
    \end{align*}
    and similarly
    \begin{align*}
        \int_{cN^{-\frac{1}{s}}}^c\mu(B(x,r))r^{-s-1}\dd r&\geq \sum_{n=1}^{N-1}\mu(B(x,c(n+1)^{-\frac{1}{s}}))\int_{c(n+1)^{-\frac{1}{s}}}^{cn^{-\frac{1}{s}}}r^{-s-1}\dd r\\
        &=\frac{1}{c^ss}\sum_{n=1}^{N}\mu(B(x,cn^{-\frac{1}{s}}))-\frac{1}{c^ss}\mu(B(x,c)).
    \end{align*}
    Noting that if $c\leq 1$, then
    \begin{equation*}
        \int_{cN^{-\frac{1}{s}}}^1\mu(B(x,r))r^{-s-1}\dd r-C\leq\int_{cN^{-\frac{1}{s}}}^c\mu(B(x,r))r^{-s-1}\dd r\leq \int_{cN^{-\frac{1}{s}}}^1\mu(B(x,r))r^{-s-1}\dd r,
    \end{equation*}
    for some constant $C\geq 0$ and a similar inequality holds in the case $c>1$, by multiplying both sides by $c^s$, dividing by $\log cN^{-\frac{1}{s}}$ and taking $N\to\infty$ gives
    \begin{equation}\label{eq:average-limit}
        \lim_{N\to\infty}\left|\frac{1}{\log c^{-s}N}\sum_{n=1}^{N}\mu(B(x,cn^{-\tfrac{1}{s}}))-c^s\AA(x,cN^{-\frac{1}{s}})\right|
        =0.
    \end{equation}
    Note that it follows from the definition of $\overline{n}(\rrr,\mtt{i}_x,k)$ that there is a constant $C>0$, such that
    \begin{equation}\label{eq:n-bounds}
        C^{-1}\|g_{\mtt{i}_x|_{k}}'\|^{-s}\leq \overline{n}(\rrr,\mtt{i}_x,k)\leq C\|g_{\mtt{i}_x|_{k}}'\|^{-s}.
    \end{equation}
    Taking the limit in \cref{eq:average-limit} along the subsequence $\overline{n}(\rrr,\mtt{i}_x,k)$ and noting that
    \begin{equation*}
        \lim_{k\to\infty}\frac{\log C\|g_{\mtt{i}_x|_{k}}'\|^{-s}}{\log \overline{n}(\rrr,\mtt{i}_x,k)}=\lim_{k\to\infty}\frac{\log \|g_{\mtt{i}_x|_{k}}'\|^{-s}}{\log \overline{n}(\rrr,\mtt{i}_x,k)}=\lim_{k\to\infty}\frac{\log c^{-s}\overline{n}(\rrr,\mtt{i}_x,k)}{\log \overline{n}(\rrr,\mtt{i}_x,k)}=1,
    \end{equation*}
    gives the claim.
\end{proof}
We note that since \cref{eq:n-bounds} and the bounded distortion principle imply that there is a constant $C>0$, such that $ \overline{n}(\rrr,\mtt{i}_x,k+1)\geq C \overline{n}(\rrr,\mtt{i}_x,k)$, it follows from \cref{l:avg-density-sum} that
\begin{equation}\label{eq:lower-avg-density}
    \liminf_{k\to\infty}\frac{1}{-s\log \|g_{\mtt{i}_x|_{k}}'\|}\sum_{n=1}^{\overline{n}(\rrr,\mtt{i}_x,k)}\mu(B(x,cn^{-\tfrac{1}{s}}))=c^s\underline{\AA}(x),
\end{equation}
and a similiar equality holds for $c^s\overline{\AA}(x)$ by replacing the liminf by limsup.

Let us set
\begin{equation*}
    \alpha_{\min}=\inf_{x\in X}\AA(x),\quad \text{and}\quad \alpha_{\max}=\sup_{x\in X}\AA(x).
\end{equation*}
It is well known that there is a constant $\alpha_{\min}<\alpha_0<\alpha_{\max}$, such that $\AA(x)=\alpha_0$, for $\mu$-almost every $x\in X$, see e.g. \cite[Theorem 6.7]{Falconer1997}. For $\alpha\in\R$, we let
\begin{equation*}
    \Delta(\alpha)=\{x\in X\colon \AA(x)=\alpha\},\quad\text{and}\quad\underline{\Delta}(\alpha)=\{x\in X\colon \underline{\AA}(x)\leq\alpha\}.
\end{equation*}
The critical constant in the Dvoretzky problem turns out to depend subtly on the \emph{multifractal spectrum} $f(\alpha)\coloneqq \dimh \Delta(\alpha)$ of the average densities. To prove \cref{thm:self-conformal}, we need the following proposition, which shows that the Hausdorff dimension of $\underline{\Delta}(\alpha)$ agrees with $f(\alpha)$ for values $\alpha\leq \alpha_0$. This observation is standard in the context of multifractal analysis for local dimensions. We were however unable to find a result in the literature which we could apply for the average densities, so we indicate the changes needed to derive the proposition using methods of  \cite{Barreira2013}, whose framework multifractal analysis for average densities of self-conformal measures falls into. Since this proof is quite different in flavour to the rest of the article, and in order to not disrupt the presentation, we defer the proof to \cref{sec:multifractal}.
\begin{proposition}\label{prop:sub-level-dim}
    For $\alpha\leq \alpha_0$, we have
    \begin{equation*}
        \dimh\underline{\Delta}(\alpha)=f(\alpha),
    \end{equation*}
    and for $\alpha\geq\alpha_0$, $\dimh\underline{\Delta}(\alpha)=s$.
\end{proposition}
We note that by \cite[Theorem 1]{Barreira2013}, the function $f\colon (\alpha_{\min},\alpha_{\max})\to\R$ is continuous and achieves it's maximum $s$ at $\alpha=\alpha_0$. We are now ready to prove \cref{thm:self-conformal}.
\begin{proof}[Proof of \cref{thm:self-conformal}]
    Let us start by proving \cref{it:sc-no-cover}. Let $\alpha$ be a real number which maximizes $\frac{f(\alpha)}{\alpha}$. Since $\dimh \Delta(\alpha)=f(\alpha)$, for any $0<p<f(\alpha)$, there exists a measure $\nu\in\mathcal{P}(\Delta(\alpha))$ satisfying $\dimh\nu > p$. This implies in particular that $\nu$ has no atoms. Choose now $0<p<f(\alpha)$ large enough, such that there exists $t>0$, which satisfies $c^s<t<\frac{p}{s\alpha}$ and then let $\varepsilon>0$, be small enough that
    \begin{equation*}
        d\coloneqq p-st\alpha-st\varepsilon>0.
    \end{equation*}
    By \cref{l:annulus-dim}, for $\nu$-almost every $x\in X$, for all large enough $k\in\N$, we have
    \begin{equation*}
        \nu(A_k(x))\leq  \|g_{\mtt{i}_x|_{k}}'\|^{p}.
    \end{equation*}
    Combining this with \cref{eq:lower-avg-density}, and noting that $\AA(x)$ exists for all $x$ in the support of $\nu$, we see that for $\nu$-almost every $x\in X$, for all large enough $k$ we have
    \begin{align*}
        \nu(A_k(x))\exp\sum_{n=1}^{\overline{n}(\rrr,\mtt{i}_x,k)}\mu(B(x,cn^{-\frac{1}{s}}))&\leq \|g_{\mtt{i}_x|_{k}}'\|^{p - sc^s(\alpha+\varepsilon)}\\
        &\leq \|g_{\mtt{i}_x|_{k}}'\|^{p - st\alpha-st\varepsilon}\leq\lambda_{\max}^{dk}.
    \end{align*}
    In particular
    \begin{equation*}
        \sum_{k=1}^{\infty}\nu(A_k(x))\exp\sum_{n=1}^{\overline{n}(\rrr,\mtt{i},k)}\mu(B(x,r_n)))<\infty,
    \end{equation*}
    for $\nu$ almost every $x$ and \cref{it:sc-no-cover} follows by combining \cref{prop:pw-billard} and \cref{thm:main} and recalling that $X=X_{\mu,\rrr}$.

    For \cref{it:sc-cover} let $\nu\in\mathcal{P}(X)$ and denote by $t=\nu\text{-}\esssup \underline{\AA}(x)$. Let us first assume that $t<\alpha_0$. Note that $\nu(\underline{\Delta}(t))=1$, and therefore $\overline{\dimh} \nu\leq f(t)$ by \cref{prop:sub-level-dim}. In particular, by \cref{l:annulus-dim}, for any $p>f(t)$ and for $\nu$-almost every $x\in X$, there exists an increasing sequence $(k_m)_m$ of natural numbers, such that
    \begin{equation*}
        \nu(A_{k_m}(x))\geq  \|g_{\mtt{i}_x|_{k_m}}'\|^{p},
    \end{equation*}
    for all $m\in\N$. We note that the sequence $(k_m)_m$ is alowed to depend on $x$ but this does not cause problems in the proof so we suppress this dependence from our notation. Now let $p>f(t)$ be small enough that we may choose $q>0$ which satisfies $c^s>q>\frac{p}{st}$ and then $\varepsilon>0$, such that
    \begin{equation*}
        d\coloneqq p-sqt +sq\varepsilon<0.
    \end{equation*}
    Note that by definition $\underline{\AA}(x)\geq t-\varepsilon$ in a set of positive $\nu$-measure. Therefore, for $\nu$-positively many $x$, for all large enough $m\in\N$, we have
    \begin{align*}
        \nu(A_{k_m}(x))\exp\sum_{n=1}^{\overline{n}(\rrr,\mtt{i}_x,k_m)}\mu(B(x,cn^{-\frac{1}{s}}))&\geq \|g_{\mtt{i}_x|_{k_m}}'\|^{p - sc^s(t-\varepsilon)}\\
        &\geq \|g_{\mtt{i}_x|_{k_m}}'\|^{p-sqt +sq\varepsilon}\geq\lambda_{\max}^{dk_m}\geq 1.
    \end{align*}
    Thus
    \begin{equation}\label{eq:inf-pw}
        \sum_{k=1}^{\infty}\nu(A_k(x))\exp\sum_{n=1}^{\overline{n}(\rrr,\mtt{i},k)}\mu(B(x,r_n)))=\infty,
    \end{equation}
    for $\nu$-positively many $x$.

    On the other hand, if $t\geq\alpha_0$, we have
     \begin{equation*}
         c^s>\max_{\alpha}\frac{f(\alpha)}{s\alpha}\geq \frac{f(\alpha_0)}{s\alpha_0}=\frac{1}{\alpha_0}\geq \frac{1}{t}.
     \end{equation*}
     Moreover, we again have that $\nu(\underline{\Delta}(t))=1$, and by \cref{prop:sub-level-dim}, $\overline{\dimh} \nu\leq s$. By estimating as before with $s$ in place of $f(t)$, we see that \cref{eq:inf-pw} holds for $\nu$ positively many $x$ in this case as well. The claim then follows from \cref{prop:pw-billard} and \cref{thm:main}.
\end{proof}
\begin{remark}\label{rem:cantor}
If $X$ is a homogeneous self-similar set, that is $g_i(x)=ax+t_i$, for every $i\in\Lambda$, for example if $X$ is the classical Cantor set, then the function $f(\alpha)$ is a concave real analytic function on $[\alpha_{\min},\alpha_{\max}]$. In particular, since $f'(\alpha_0)=0$, and since $\frac{f(\alpha)}{\alpha}$ (and therefore $\left(\frac{f(\alpha)}{s\alpha}\right)^{\frac{1}{s}})$ is maximised when $f'(\alpha)=\frac{f(\alpha)}{\alpha}$, we see that the maximum is achived at some $\alpha<\alpha_0$. This is particularly interesting since, by \cref{prop:pw-billard} (or rather it's proof), this implies that there are constants $c>0$, such that for $\rrr=(cn^{-\frac{1}{s}})$, we have
\begin{equation*}
    I_{\mu,\rrr}(\mu)=\infty,
\end{equation*}
but $I_{\mu,\rrr}(\nu)<\infty$ for some $\nu\in\PP(X)$. Unlike for the Lebesgue measure, one therefore cannot replace the condition $\capa_{\mu,\rrr}(X)=0$ with $I_{\mu,\rrr}(\mu)=\infty$ even for very regular measures like the natural measure on the Cantor set.
\end{remark}

\section{Average densities of the Cantor measure}\label{sec:cantor}
In this final section of the paper, we study the Dvoretzky covering problem driven by the natural measure on the Cantor set $X$. This falls into the framework of \cref{sec:self-conformal}, since $X$ is the attractor of the IFS $\{g_1(x)=\frac{1}{3}x,g_2(x)=\frac{1}{3}x+\frac{2}{3}\}$ and the natural measure $\mu$, which is the unique Borel probability measure which satisfies
\begin{equation*}
    \mu=\frac{1}{2}(g_1)_*\mu+\frac{1}{2}(g_2)_*\mu,
\end{equation*}
is in the language of \cref{sec:self-conformal} the $s$-conformal measure on $X$, with $s=\frac{\log2}{\log 3}$.

Our aim in this section is to establish concrete bounds for the critical constant in the Dvoretzky covering problem for the sequences $\rrr=(cn^{-\frac{1}{s}})$ and in particular, to prove \cref{prop:cantor}. Recall that by \cref{thm:self-conformal}, this critical constant for the covering problem is given by the quantity $\left(\max_{\alpha}\frac{f(\alpha)}{s\alpha}\right)^{\frac{1}{s}}$. For the natural measure on the Cantor set, the $\mu$-almost sure value for the average density of $\mu$ can be numerically calculated, see e.g. \cite[Theorem 6.5.]{Falconer1997} (Note that \cite{Falconer1997} uses different normalisation than us), and it is known that
\begin{equation*}
    \AA(x)\eqqcolon\alpha_0=0.9654\ldots,
\end{equation*}
for $\mu$-almost every $x$. In particular, it follows from \cref{thm:self-conformal} that the trivial bound
\begin{equation*}
    c<\left(\frac{1}{\alpha_0}\right)^{\frac{1}{s}}=1.0573\ldots,
\end{equation*}
is enough for $X$ to be almost surely not covered by $E_{\rrr}$. Recall that by the simple convexity argument in \cref{rem:cantor}, this value is not the actual critical value. As the main result of this section, we give explicit upper and lower bounds for the critical constant, and in particular, the lower bound beats the trivial bound given by $\alpha=\alpha_0$. Note that \cref{prop:cantor} follows immediately from the following proposition together with \cref{thm:self-conformal}.
\begin{proposition}\label{prop:maximising-alpha}
    If $\mu$ is the natural measure on $X$, then
    \begin{equation*}
        \left(\frac{1}{\alpha_0}\right)^{\frac{1}{s}}<1.06126\leq \left(\max_{\alpha}\frac{f(\alpha)}{s\alpha}\right)^{\frac{1}{s}}\leq \left(\frac{1}{\alpha_{\min}}\right)^{\frac{1}{s}}\leq 1.37546.
    \end{equation*}
\end{proposition}

\subsection{Average density at \texorpdfstring{$0$}{0} and the upper bound}
The upper bound in \cref{prop:maximising-alpha} follows rather easily from the fact that $\AA(x)$ is minimised at $x=0$. Going forward, we denote by $T\colon X\to X$ the unique dynamics on $X$ conjugated with $\sigma$ by the natural projection map $\pi$, that is $T$ is the function defined by
\begin{equation*}
    T\circ \pi(\mtt{i})=\pi\circ \sigma(\mtt{i}),
\end{equation*}
for all $\mtt{i}\in\Sigma$. This is well defined since $X$ satisfies the SSC, and clearly $\mu$ is $T$-invariant by definition. Let us also extend the natural projection map to finite words by setting 
\begin{equation*}
    \pi(\mtt{i})=\varphi_{\mtt{i}}(0)=\sum_{j=1}^{|\mtt{i}|}2i_k3^{-j},
\end{equation*}
for all $\mtt{i}\in\Sigma_*$. The following lemma gives the lower bound in \cref{prop:maximising-alpha}.
\begin{lemma}\label{lemma:alpha-min-bound}
    Let $\mu$ be the natural measure on the Cantor set $X$. Then
    \begin{equation*}
        \alpha_{\min}=\AA(0)\geq 0.81781
    \end{equation*}
\end{lemma}
\begin{proof}
    By the self-similarity of $\mu$, and a simple change of variables, we have
    \begin{align}\label{eq:avg-density-0}
        \int_{3^{-k}}^1\mu(B(x,r))r^{-s-1}\dd r&=\sum_{j=0}^{k-1}\int_{3^{-(j+1)}}^{3^{-j}}\mu(B(x,r))r^{-s-1}\dd r\\
        &=\sum_{j=0}^{k-1}\int_{3^{-1}}^{1}\mu(B(T^jx,r))r^{-s-1}\dd r\nonumber.
    \end{align}
    Noting that if $r\geq 3^{-1}$, then if $x\in [0,3^{-1}]\cap X$, then $B(0,r)\cap X\subset B(x,r)\cap X$ and if $x\in [2\cdot3^{-1},1]\cap X$, then $B(1,r)\cap X\subset B(x,r)\cap X$. Since $\mu(B(0,r))=\mu(B(1,r))$ by symmetry, we have that $\mu(B(0,r))\leq \mu(B(x,r))$, for all $x\in X$. Since $T0=0$, this shows that $\alpha_{\min}=\AA(0)$, provided the limit exists.

    To see that the limit exists, and to numerically estimate the value of $\AA(0)$, notice that by using \cref{eq:avg-density-0} at $x=0$, and the fact that $T0=0$, we have
    \begin{equation*}
        \int_{3^{-k}}^1\mu(B(0,r))r^{-s-1}\dd r=k\int_{3^{-1}}^{1}\mu(B(0,r))r^{-s-1}\dd r,
    \end{equation*}
    and therefore,
    \begin{equation*}
        \AA(0)=\frac{1}{\log 3}\int_{3^{-1}}^{1}\mu(B(0,r))r^{-s-1}\dd r.
    \end{equation*}
    If we denote by $\Sigma^1_k=\{1\mtt{i}\colon \mtt{i}\in\Sigma_k\}$, then by Fubini, we have
    \begin{equation*}
        \int_{3^{-1}}^{1}\mu(B(0,r))r^{-s-1}\dd r=\frac{1}{s}\int_{3^{-1}}^1x^{-s}\dd \mu(x)=\frac{1}{s}\sum_{\mtt{i}\in\Sigma_k^1}\int_{g_{\mtt{i}}(X)}x^{-s}\dd\mu(x).
    \end{equation*}
    Notice moreover that if $x\in g_{\mtt{i}}(X)$, then $x\leq \pi(\mtt{i})+3^{-(k+1)}$, and therefore
    \begin{equation*}
        \frac{1}{s}\sum_{\mtt{i}\in\Sigma_k^1}\int_{g_{\mtt{i}}(X)}x^{-s}\dd\mu(x)\geq \frac{1}{s}\sum_{\mtt{i}\in\Sigma_k^1}2^{-(k+1)}(\pi(\mtt{i})+3^{-(k+1)})^{-s}.
    \end{equation*}
    A computer assisted calculation with $k=5$ gives the lower bound
    \begin{equation*}
        \alpha_{\min}=\AA(0)\geq 0.81781.
    \end{equation*}
\end{proof}

\subsection{\texorpdfstring{$n$}{n}-step Bernoulli measures and the lower bound}
The proof of the lower bound in \cref{prop:maximising-alpha} relies on the fact that for many ergodic measures $\nu$ on $X$, the $\nu$-typical value of $\AA(x)$ can be estimated numerically. If $\nu\in\PP(X)$ is then a measure which satisfies
\begin{equation*}
    \AA(x)=\alpha_{\nu},
\end{equation*}
for $\nu$-almost all $\alpha$, then $f(\alpha_{\nu})\geq \dimh\nu$ and if we manage to find an upper bound $\alpha_{\nu}\leq \alpha'$, then
\begin{equation}\label{eq:constant-lower-bound}
    \max_{\alpha}\frac{f(\alpha)}{s\alpha}\geq \frac{f(\alpha_{\nu})}{s\alpha_{\nu}}\geq \frac{\dimh\nu}{s\alpha'}.
\end{equation}
To prove \cref{prop:maximising-alpha}, it then suffices to find a single measure $\nu\in\PP(X)$, for which $\AA(x)$ is a constant $\nu$-almost everywhere, and for which we can calculate the Hausdorff dimension and give rigorous upper bounds for the $\nu$-typical value of the average density. The measures we focus on are the so called $n$-step Bernoulli measures on $X$.

Recall that we denote $\Lambda=\{0,1\}$ and note that for any $n\in\N$, the Cantor set $X$ is the attractor of the  IFS $\{g_{\mtt{i}}\colon \mtt{i}\in\Lambda^n\}$. We say that a measure $\nu\in\PP(X)$ is an \emph{$n$-step Bernoulli measure}, if it is a \emph{self-similar measure} for the IFS $\{\varphi_{\mtt{i}}\colon \mtt{i}\in\Lambda^n\}$, that is if there is a probability vector $(p_{\mtt{i}})_{\mtt{i}\in\Lambda^n}$, such that
\begin{equation*}
    \nu=\sum_{\mtt{i}\in\Lambda^n}p_{\mtt{i}}(g_{\mtt{i}})_*\nu.
\end{equation*}
Note that an $n$-step Bernoulli measure $\nu$ is ergodic under $T^n$.  A useful fact to us, which easily follows from Birkhoff's ergodic theorem is that the Hausdorff dimension of an $n$-step Bernoulli measure can easily be calculated explicitly: If $\nu$ is an $n$-step Bernoulli measure associated with the probabilty vector $(p_{\mtt{i}})_{\mtt{i}\in\Lambda^n}$, then
\begin{equation}\label{eq:n-step-dimension}
    \dimh \nu=\frac{-\sum_{\mtt{i}\in\Lambda^n}p_{\mtt{i}}\log p_{\mtt{i}}}{n\log 3}.
\end{equation}
Moreover, the average density of $\mu$ is a constant $\nu$-almost everywhere, and this constant can be estimated numerically. To see how to do this, let us set
\begin{equation*}
    \phi_k^n(x)=\int_{3^{-n(k+1)}}^{3^{-nk}}\mu(B(x,r))r^{-s-1}\dd r.
\end{equation*}
Since the measure $\mu$ is invariant under $T$, and therefore under $T^n$, we have that
\begin{equation*}
    \phi_k^n(x)=\phi_{k-1}^n(T^n x)=\ldots=\phi_{0}^n(T^{kn} x).
\end{equation*}
Since $\nu$ is ergodic with respect to $T^n$, an application of Birkhoff's ergodic theorem gives
\begin{align*}
    \int\phi_0^n(y)\dd\nu(y)&=\lim_{k\to\infty}\frac{1}{k}\sum_{j=0}^{k-1}\phi_j^n(x)=\lim_{k\to\infty} \frac{1}{k}\int_{3^{-nk}}^1\mu(B(x,r))r^{-s-1}\dd r=n\log 3\AA(x),
\end{align*}
for $\nu$-almost every $x\in X$. Moreover, a simple calculation shows that
\begin{align*}
    \int&\phi_0^n(y)\dd\nu(y)=\int\int_{3^{-n}}^{1}\mu(B(y,r))r^{-s-1}\dd r \dd\nu(y)\\
    &=\int\int_{3^{-n}}^{1}\int\chi_{\{|x-y|\leq r\}}(x)r^{-s-1}\dd\mu(x)\dd r \dd\nu(y)\\
    &=\int\int\int_{\max\{|x-y|,3^{-n}\}}^{1}r^{-s-1}\dd r \dd\mu(x)\dd\nu(y)\\
    &=\iint_{|x-y|\leq 3^{-n}}\int_{3^{-n}}^{1}r^{-s-1}\dd r \dd\mu(x)\dd\nu(y)+\iint_{|x-y|\geq 3^{-n}}\int_{|x-y|}^{1}r^{-s-1}\dd r \dd\mu(x)\dd\nu(y)\\
    &=\frac{1}{s}\iint_{|x-y|\leq 3^{-n}}(2^n-1)\dd\mu(x)\dd\nu(y)+\frac{1}{s}\iint_{|x-y|\geq 3^{-n}}(|x-y|^{-s}-1) \dd\mu(x)\dd\nu(y)\\
    &=\frac{1}{s}\iint_{|x-y|\geq 3^{-n}}|x-y|^{-s} \dd\mu(x)\dd\nu(y),
\end{align*}
where in the last step we used the fact that $\mu\times \nu(\{(x,y)\in X\times X\colon |x-y|\leq 3^{-n}\})=2^{-n}$. By combining the previous two calculations, we have that if $\nu$ is an $n$-step Bernoulli measure on $X$, then
\begin{equation}\label{eq:nu-typical-density}
    \AA(x)=\frac{1}{n\log 2}\iint_{|x-y|\geq 3^{-n}}|x-y|^{-s} \dd\mu(x)\dd\nu(y),
\end{equation}
for $\nu$-almost every $x$. 

To estimate the integral in \cref{eq:nu-typical-density} numerically, we denote by
\begin{equation*}
    \Xi_n=\{(\mtt{i},\mtt{j})\in\Lambda^n\times \Lambda^n\colon \mtt{i}\ne\mtt{j}\},
\end{equation*}
and note that $\pi(\Xi_n)$ is the domain of integration in \cref{eq:nu-typical-density}. For $\mtt{i}\in\Sigma_*$, and any $k\in\N$, we let $\mtt{i}\Lambda^{k}=\{\mtt{i}\mtt{j}\colon \mtt{j}\in\Lambda^{k}\}$. The following proposition allows us to bound the average density from above at $\nu$-typical points.
\begin{proposition}\label{prop:density-computational}
    Let $\mu$ be the natural measure on $X$ and let $\nu\in\PP(X)$ be an $n$-step Bernoulli measure. Then for $\nu$-almost every $x\in X$, we have
    \begin{equation}\label{eq:density-computational}
        \AA(x)\leq \frac{1}{n\log 2}\sum_{(\mtt{i},\mtt{j})\in\Xi_n}\sum_{\mtt{k}\in\mtt{i}\Lambda^{kn}}\sum_{\mtt{l}\in\mtt{j}\Lambda^{kn}}2^{-(k+1)n}p_{\mtt{l}}\left(|\pi(\mtt{k})-\pi(\mtt{l})|-3^{-(k+1)n}\right)^{-s}.
    \end{equation}
\end{proposition}
\begin{proof}
    Note that
    \begin{align*}
    \iint_{|x-y|\geq 3^{-n}}|x-y|^{-s} \dd\mu(x)\dd\nu(y)&=\sum_{(\mtt{i},\mtt{j})\in\Xi_n}\int_{\varphi_{\mtt{i}}(X)}\int_{\varphi_{\mtt{j}}(X)}|x-y|^{-s} \dd\mu(x)\dd\nu(y)\\
    &=\sum_{(\mtt{i},\mtt{j})\in\Xi_n}\sum_{\mtt{k}\in\mtt{i}\Lambda^{kn}}\sum_{\mtt{l}\in\mtt{j}\Lambda^{kn}}\int_{\varphi_{\mtt{k}}(X)}\int_{\varphi_{\mtt{l}}(X)}|x-y|^{-s} \dd\mu(x)\dd\nu(y).
    \end{align*}
    If $x\in \varphi_{\mtt{k}}(X)$ and $y\in \varphi_{\mtt{l}}(X)$ with $\mtt{k}\in\mtt{i}\Lambda^{kn}$ and $\mtt{l}\in\mtt{j}\Lambda^{kn}$, then, since $\mtt{i}\ne\mtt{j}$, we have
    \begin{equation}\label{eq:dist-lb}
        |x-y|\geq |\pi(\mtt{k})-\pi(\mtt{l})|-3^{-(k+1)n}.
    \end{equation}
    The claim follows by observing that $\mu(\varphi_{\mtt{k}}(X))=2^{-(k+1)n}$ and $\nu(\varphi_{\mtt{l}}(X))=p_{\mtt{l}}$.
\end{proof}
\begin{remark}
    Using the trivial upper bound in place of \cref{eq:dist-lb}, we have
    \begin{equation*}\label{eq:density-computational'}
        \AA(x)\geq \frac{1}{n\log 2}\sum_{(\mtt{i},\mtt{j})\in\Xi_n}\sum_{\mtt{k}\in\mtt{i}\Lambda^{kn}}\sum_{\mtt{l}\in\mtt{j}\Lambda^{kn}}2^{-(k+1)n}p_{\mtt{l}}\left(|\pi(\mtt{k})-\pi(\mtt{l})|+2\cdot3^{-(k+1)n}\right)^{-s},
    \end{equation*}
    from which it is easy to see that, for any $n\in\N$ and for $\nu$-almost every $x\in X$, we have
    \begin{equation*}
        \AA(x)=\lim_{k\to\infty} \AA_{n,k}(x).
    \end{equation*}
    One can also use these estimates to derive explicit error bounds for $\AA_{n,k}(x)$, but since upper bounds suffice to us and since the constants we find are not sharp anyway, we will not pursue this.
\end{remark}
\begin{proof}[Proof of \cref{prop:maximising-alpha}]
The upper bound is immediate from \cref{lemma:alpha-min-bound}. For the lower bound, let $\nu$ denote the $3$-step Bernoulli measure associated with the probabilities
\begin{align*}
    &p_{000} = p_{111} = 0.1431125\\
    &p_{001} = p_{110} = 0.1243875\\
    &p_{010} = p_{101} = 0.1081125\\
    &p_{011} = p_{100} = 0.1243875,
\end{align*}
The method for choosing the weights is purely heuristic and was done by trial and error: The points which are close to the edges of construction cylinders at many scales should have a low average density and therefore should be more difficult to cover. Thus we concentrate more mass on the edge cylinders. However, concentrating too much mass on the edges lowers the Hausdorff dimension of the measure, and balancing these competing quantities seems to be rather delicate.

Using \cref{prop:density-computational}, a computer assisted calculation, with $k=3$, gives us the upper bound
\begin{equation*}
    \AA(x)\leq 0.96091,
\end{equation*}
for $\nu$-almost every $x\in X$. Plugging the probabilities $p_{\mtt{i}}$ into \cref{eq:n-step-dimension} and using observation \cref{eq:constant-lower-bound}, we get
\begin{equation*}
    \left(\max_{\alpha}\frac{f(\alpha)}{s\alpha}\right)^{\frac{1}{s}}\geq 1.06126,
\end{equation*}
which proves \cref{prop:maximising-alpha}.
\end{proof}

\appendix
\section{Multifractal analysis for asymptotically additive potentials}\label{sec:multifractal}
In this appendix, we indicated the changes needed in the proof of \cite[Theorem 1 (1)]{Barreira2013} to obtain \cref{prop:sub-level-dim}. To stay within the theme of the article, we will present the result in the context of self-conformal IFSs but we note that the same proof works in the context of \cite{Barreira2013}, that is for saturated dynamics with an upper semicontinuous entropy map. Let $\{g_i\}_{i\in \Lambda}$ be a strongly separated self-conformal IFS whose attractor we continue denoting by $X$. We equip $X$ with the dynamics $T\colon X\to X$, which is the unique continuous function on $X$ conjugated to the left shift $\sigma$ on the underlying symbolic space $\Sigma$ by the natural projection map $\pi$, that is $T$ is the unique map which satisfies
\begin{equation*}
    T\circ\pi(\mathtt{i})=\pi\circ\sigma (\mathtt{i}).
\end{equation*}
The map $T$ is well defined on $X$ since $\pi$ is a bijection by the SSC.

We let $\mathcal{M}(X)$ denote the space of Borel probability measures supported on $X$ and $\mathcal{M}(X,T)$ denote the space of $T$-invariant Borel probability measures on $X$. Recall that both $\mathcal{M}(X)$ and $\mathcal{M}(X,T)$ are compact in the weak-$\ast$ topology. For $\nu\in\mathcal{M}(X,T)$, we denote by $h_{\nu}(T)$ the standard measure theoretic entropy of $T$ with respect to $\nu$; for the definition and basic properties we refer the reader to \cite{Walters1982}.

We say that a sequence $\Phi=\{\varphi_n\}$ of functions $\varphi_n\colon X\to \R$ is \emph{asymptotically additive} if for every $\varepsilon>0$ there exists a continuous function $\varphi\colon X\to \R$ such that
\begin{equation*}
    \limsup_{n\to\infty}\frac{1}{n}\sup_{x\in X}|\varphi_n(x)-S_n\varphi(x)|<\varepsilon,
\end{equation*}
where $S_n\varphi(x)=\sum_{k=1}^{n-1}\varphi\circ T^k$. If $\varphi_n=S_n\varphi(x)$ for all $n$ then $\Phi$ is said to be \emph{additive}. For a continuous function $\psi$ and $F\subset X$, we denote by $P_F(\psi)$ the \emph{topological pressure of $\psi$ on $F$}. The definition is somewhat technical and we will not use it in our proof so to keep this section concise we instead refer the reader to \cite[Section 2.3]{Barreira2013} for the definition. For an asymptotically additive sequence $\Phi$ and $\nu\in\mathcal{M}(X,T)$ we define
\begin{equation*}
    \Phi_*(\nu)=\lim_{n\to\infty}\frac{1}{n}\int \varphi_n\dd\nu.
\end{equation*}

We will formulate our result for the general notion of $u$-dimension, although for our application the special case of Hausdorff dimension would be enough. Let $u\colon X\to \R^+$ be a continuous function. For every $\mtt{i}\in\Sigma_n$, we define
\begin{equation*}
    u(\mtt{i})=\sup\left\{\sum_{k=0}^{n-1}u(T^k x)\colon x\in g_{\mtt{i}}(X)\right\}.
\end{equation*}
Given a set $F\subset X$ and $\alpha\in\R$ we let
\begin{equation*}
    N(F,\alpha,u)=\lim_{n\to\infty}\inf_{\Gamma}\sum_{\mtt{i}\in\Gamma}\exp(-\alpha u(\mtt{i})),
\end{equation*}
where the infimum is taken over all countable collections $\Gamma\subset \bigcup_{k\geq n}\Sigma_k$ which satisfy $F\subset \bigcup_{\mtt{i}\in\Gamma}g_{\mtt{i}}(X)$. The \emph{$u$-dimension} of $F$ is defined by
\begin{equation*}
    \dim_u F=\inf\{\alpha\in\R\colon N(F,\alpha,u)=0\}.
\end{equation*}
There is an intimate connection between the topological pressure and the $u$-dimension.
\begin{proposition}[{\cite[Proposition 3]{Barreira2013}}]\label{prop:udim-pressure}
    We have $\dim_u F=s$, where $s$ is the unique root of the equation
    \begin{equation*}
        P_F(-s u)=0.
    \end{equation*}
\end{proposition}

Let now $\Phi$ and $\Psi$ be asymptotically additive sequences, and assume that there exists a constant $\gamma>0$, such that $\varphi_n(x),\psi_n(x)\geq n\gamma$, for all $n\in\N$ and $x\in X$. For each $\alpha\in\R$ we let
\begin{equation*}
    \Delta(\Phi,\Psi,\alpha)=\{x\in X\colon \lim_{n\to\infty}\frac{\varphi_n(x)}{\psi_n(x)}=\alpha\},
\end{equation*}
and 
\begin{equation*}
    \underline{\Delta}(\Phi,\Psi,\alpha)=\{x\in X\colon \liminf_{n\to\infty}\frac{\varphi_n(x)}{\psi_n(x)}\leq \alpha\}.
\end{equation*}
The following proposition---which is a modification of \cite[Theorem 5]{Barreira2013}---is a variational principle for the $u$-dimension for the sub-level sets.
\begin{proposition}\label{prop:udim}
    For each $\alpha\in\R$, we have
    \begin{equation*}
        \dim_u\underline{\Delta}(\Phi,\Psi,\alpha)=\max\left\{\frac{h_{\nu}(T)}{\int u\dd \nu}\colon \Phi_*(\nu)\leq \alpha \Psi_*(\nu),\,\nu\in\mathcal{M}(X,T)\right\}.
    \end{equation*}
\end{proposition}
\begin{proof}
    Our proof closely follows the proof of \cite[Theorem 1]{Barreira2013}. Notice first that by \cref{prop:udim-pressure}, it suffices to show that $P_{\underline{\Delta}(\Phi,\Psi,\alpha)}(-su)=D_{\alpha}$, where 
    \begin{equation*}
        D_\alpha=\max\left\{h_{\nu}(T)-\int su\dd\nu\colon \Phi_*(\nu)\leq\alpha\Psi_*(\nu),\,\nu\in\mathcal{M}(X,T)\right\}.
    \end{equation*}
    
    Given $x\in X$, we denote by $V(x)\subset \mathcal{M}(X,T)$ the accumulation points of the sequence
    \begin{equation*}
        \nu_{x,n}\coloneqq \frac{1}{n}\sum_{j=0}^{n-1}\delta_{T^j(x)}.
    \end{equation*}
    The following is \cite[Lemma 6]{Barreira2013}.
    \begin{lemma}\label{lemma:pressure-upper-bound}
        For each $t\in\R$ let
        \begin{equation*}
            R(t)=\{x\in X\colon h_{\nu}(T)-\int su\dd\nu\leq t\text{ for some }\nu\in V(x)\}.
        \end{equation*}
        Then $P_{R(t)}(-su)\leq t$.
    \end{lemma}

    Let us now take $x\in\underline{\Delta}(\Phi,\Psi,\alpha)$ and let $(n_k)_k$ be an increasing sequence of integers which satisfies
    \begin{equation}\label{eq:conv-subseq}
        \lim_{k\to\infty}\frac{\varphi_{n_k}(x)}{\psi_{n_k}(x)}\leq \alpha.
    \end{equation}
    Let $\nu$ be an accumulation point of the sequence
    \begin{equation*}
        \nu_{x,k}\coloneqq \frac{1}{n_k}\sum_{j=0}^{n_k-1}\delta_{T^j(x)},
    \end{equation*}
    at least one of which exists since $\mathcal{M}(X)$ is compact. Clearly $\nu\in V(x)$ and by relabeling the indices we may assume without loss of generality that $\nu_{x,k}\to\nu$ in the weak-$\ast$ topology. It was shown in \cite[p. 210]{Barreira2011} that
    \begin{equation*}
        \Phi_*(\nu)=\lim_{k\to\infty} \frac{\varphi_{n_k}(x)}{n_k},
    \end{equation*}
    and the same is true for $\Psi_*(\nu)$. In particular, \cref{eq:conv-subseq} implies that
    \begin{equation*}
        \Phi_*(\nu)\leq \alpha\Psi_*(\nu).
    \end{equation*}
    Therefore by definition $h_{\nu}(T)-\int su\dd\nu\leq D_{\alpha}$ and moreover,
    \begin{equation*}
        \underline{\Delta}(\Phi,\Psi,\alpha)\subset \{x\in X\colon h_{\nu}(T)-\int su\dd\nu\leq D_{\alpha}\text{ for some }\nu\in V(x)\}.
    \end{equation*}
    Combining this with \cref{lemma:pressure-upper-bound} and \cite[Theorem 5]{Barreira2013}, we get
    \begin{equation*}
        P_{\underline{\Delta}(\Phi,\Psi,\alpha)}(-su)\leq D_{\alpha}.
    \end{equation*}
    The lower bound follows by making trivial modifications to the proof of the lower bound in \cite[Theorem 5]{Barreira2013}, and since the upper bound suffices for our application, we leave the details to the interested reader.
\end{proof}
The following corollary is immediate by setting $u=\log |DT|$, where here and hereafter $Df$ deotes the derivative of $f$, since the conformality of the IFS implies the existence of constants $c_1,c_2>0$, such that
\begin{equation*}
    c_1(\diam g_{\mtt{i}}(X))^\alpha \leq \exp(-\alpha u(\mtt{i}))\leq c_2(\diam g_{\mtt{i}}(X))^\alpha.
\end{equation*}
\begin{corollary}\label{cor:udim}
    For each $\alpha\in\R$, we have
    \begin{equation*}
        \dimh\underline{\Delta}(\Phi,\Psi,\alpha)=\max\left\{\frac{h_{\nu}(T)}{\int \log |DT|\dd \nu}\colon \Phi_*(\nu)\leq \alpha \Psi_*(\nu),\,\nu\in\mathcal{M}(X,T)\right\}.
    \end{equation*}
\end{corollary}

Let us now describe the application to the multifractal analysis of the average densities of the conformal measure on $X$. For the remainder of the section, we let $\mu$ denote the $s$-conformal measure on $X$. Recall that we are interested in the sets
\begin{equation*}
    \Delta(\alpha)=\{x\in X\colon \AA(x)=\alpha\},\quad \text{and}\quad\underline{\Delta}(\alpha)=\{x\in X\colon \underline{\AA}(x)\leq\alpha\}.
\end{equation*}
Let us denote by $f_k\colon X\to \R$, the functions
\begin{equation*}
    f_k(x)=\int_{|DT^{k+1}(x)|^{-1}}^{|DT^{k}(x)|^{-1}}\mu(B(x,r))r^{-s-1}\dd r.
\end{equation*}
Straightforward calculations show that the sequence $F=\{\sum_{k=0}^{n-1}f_k\}_{n}$ is asymptotically additive, see \cite[Section 3]{OjalaSuomalaWu2017}. Let us denote by $\log |DT|$ the additive sequence $\{\log |DT^n|\}_n$ and note that for a fixed $x\in X$,
\begin{equation*}
    \AA(x,|DT^n(x)|^{-1})=\frac{\sum_{k=0}^{n-1}f_k(x)}{\log|DT^n(x)|}.
\end{equation*}
Since the limits in the definitions of $\overline{\AA}(x)$ and $\underline{\AA}(x)$ can be calculated along any exponentially decreasing sequence, we have
\begin{equation*}
    \Delta(\alpha)=\left\{x\in X\colon \lim_{n\to\infty}\frac{\sum_{k=0}^{n-1}f_k(x)}{\log|DT^n(x)|}=\alpha\right\}=\Delta(F,\log|DT|,\alpha),
\end{equation*}
and
\begin{equation*}
    \underline{\Delta}(\alpha)=\left\{x\in X\colon \liminf_{n\to\infty}\frac{\sum_{k=0}^{n-1}f_k(x)}{\log|DT^n(x)|}\leq\alpha\right\}=\underline{\Delta}(F,\log|DT|,\alpha).
\end{equation*}
The final ingredient needed for the proof of \cref{prop:sub-level-dim} is the following simple fact from convex analysis. Recall that a function $g\colon X\to \R$, where $X$ is a vector space, is \emph{quasiconcave} if for any $x,y\in X$ and $t\in[0,1]$,
\begin{equation*}
    g(tx+(1-t)y)\geq \min\{g(x),g(y)\}.
\end{equation*}
\begin{lemma}\label{lemma:quasiconcave}
    Let $X$ be a Banach space, $Y\subset X$ be convex, and $g\colon Y\to\R$ be an upper semicontinuous quasiconcave function. Let $A\subset Y$ be compact. Then either $\max_{x\in A}g(x)=\max_{x\in Y}g(x)$ or there exists $z\in\partial_YA$ such that $g(z)=\max_{x\in A}g(x)$. Here $\partial_YA$ denotes the boundary of $A$ in the subspace topology of $Y$.
\end{lemma}
\begin{proof}
    Since $g$ is upper semicontinuous, and $A$ is compact, there exists $a\in A$, such that $g(a)=\max_{x\in A}g(x)$. Assume that there exists $y\in Y\setminus A$ such that $g(y)>g(a)$. If $a\in\partial_YA$ we have nothing to prove so we may further assume that $a\in\mathrm{int}_Y(A)$. Since $y\in\mathrm{ext}_Y(A)$, there exists $t\in[0,1]$, such that $z\coloneqq(1-t)a+ty\in\partial_YA$. Since $g(y)>g(a)$, by quasiconcavity
    \begin{equation*}
        g(z)=g((1-t)a+ty)\geq g(a),
    \end{equation*}
    so $f(z)=\max_{x\in A}g(x)$.
\end{proof}
\begin{proof}[Proof of \cref{prop:sub-level-dim}]
By applying \cref{cor:udim}, we have
\begin{equation*}
    \dimh\underline{\Delta}(\alpha)=\max\left\{\frac{h_{\nu}(T)}{\int \log|DT|\dd \nu}\colon \frac{F_*(\nu)}{\log|DT|_*(\nu)}\leq \alpha,\,\nu\in\mathcal{M}(X,T)\right\}\eqqcolon D_{\alpha},
\end{equation*}
and by \cite[Proposition 3.1]{OjalaSuomalaWu2017},
\begin{equation*}
    \dimh\Delta(\alpha)=\max\left\{\frac{h_{\nu}(T)}{\int \log|DT|\dd \nu}\colon \frac{F_*(\nu)}{\log|DT|_*(\nu)}= \alpha,\,\nu\in\mathcal{M}(X,T)\right\}\eqqcolon G_{\alpha}.
\end{equation*}
Recall that $\mathcal{M}(X,T)$ is convex and compact, and since $\nu\mapsto h_{\nu}(T)$ is affine and upper semicontinuous, see e.g. \cite[Theorems 8.1 and 8.2]{Walters1982}, and $\nu\mapsto \int \log|DS|\dd \nu$ is clearly linear and continuous, the function $\nu\mapsto\frac{h_{\nu}(T)}{\int \log|DS|\dd \nu}$ is quasiconcave and upper semicontinuous on $\mathcal{M}(X,T)$. Moreover, it is well known that the maximum of $\frac{h_{\nu}(T)}{\int \log|DS|\dd \nu}$ on $\mathcal{M}(X,T)$ is $s$ and the function has a unique maximizer on $\mathcal{M}(X,T)$ at $\mu$, see for example \cite[Theorem 4.1.8]{Barreira08} and the preceding discussion. Let us denote by
\begin{equation*}
    A=\{\nu\in\mathcal{M}(X,T)\colon \frac{F_*(\nu)}{\log|DT|_*(\nu)}\leq \alpha\}
\end{equation*}
and
\begin{equation*}
    B=\{\nu\in\mathcal{M}(X,T)\colon \frac{F_*(\nu)}{\log|DT|_*(\nu)}= \alpha\}
\end{equation*}
Since the map $\nu\mapsto\frac{F_*(\nu)}{\log|DT|_*(\nu)}$ is continuous, it is easy to verify that
\begin{equation}\label{eq:same-boundary}
    \partial_{\mathcal{M}(X,T)}B=\partial_{\mathcal{M}(X,T)}A.
\end{equation}
Note that since $\mu$ is ergodic, by the approximate ergodic theorem \cite[Corollary 6.2]{Falconer1997}, there are constants $a,\lambda\in\R$, such that for $\mu$ almost every $x\in X$,
\begin{equation*}
    \lim_{n\to\infty}\frac{1}{n}\sum_{k=0}^{n-1}f_k(x)=a,
\end{equation*}
and
\begin{equation*}
    \lim_{n\to\infty}\frac{1}{n}\log |DT^n(x)|=\lambda.
\end{equation*}
Moreover, if we let $\alpha_0$ denote the $\mu$-almost sure value of $\mathcal{A}(x)$, then $\alpha_0=a/\lambda$, and it follows from the monotone convergence theorem that
\begin{equation*}
    \frac{F_*(\mu)}{\log|DT|_*(\mu)}=\frac{\lim_{n\to\infty}\int\frac{1}{n}\sum_{k=0}^{n-1}f_k(x)\dd\mu}{\lim_{n\to\infty}\int\frac{1}{n}\log|DT^n(x)|\dd\mu}=\frac{a}{\lambda}=\alpha_0.
\end{equation*}
In particular, if $\alpha<\alpha_0$, then $\mu\not\in A\supset B$ and \cref{eq:same-boundary} together with \cref{lemma:quasiconcave} implies that $G_{\alpha}=D_{\alpha}$. On the other hand, if $\alpha\geq \alpha_0$, then $\mu\in A$, so $D_{\alpha}=s$.
\end{proof}

\bibliographystyle{amsplain}
\bibliography{bibliography}

\end{document}